\documentclass{amsart}
\usepackage{amssymb}
\usepackage[cp1250]{inputenc}
\usepackage[OT4]{fontenc}
\def\CC{\mathbb{C}}
\def\RR{\mathbb{R}}
\def\ZZ{\mathbb{Z}}
\def\NN{\mathbb{N}}
\def\DD{\mathbb{D}}
\def\BB{\mathbb{B}}
\def\TT{\mathbb{T}}
\def\CDD{\overline{\DD}}
\def\wi{\widetilde}
\def\wh{\widehat}
\def\pa{\partial}
\def\ov{\overline}
\def\CLW{\mathcal{C}^{\omega}}
\def\mb{\mathbb}
\def\mc{\mathcal}
\def\su{\subset}
\def\OO{{\mathcal O}}
\def\cC{{\mathcal C}}
\def\HH{{\mathcal H}}
\def\eps{\varepsilon}
\DeclareMathOperator{\id}{id}
\DeclareMathOperator{\wind}{wind}
\DeclareMathOperator{\re}{Re}

\DeclareMathOperator{\im}{Im}

\DeclareMathOperator{\dist}{dist}

\renewcommand{\phi}{\varphi}
\newtheorem{prop}{Proposition}[section]
\newtheorem{propp}{Proposition}[subsection]
\newtheorem{lem}[prop]{Lemma}
\newtheorem{lemm}[propp]{Lemma}
\newtheorem{tw}[prop]{Theorem}
\newtheorem{tww}[propp]{Theorem}

\newtheorem{corr}[propp]{Corollary}
\newtheoremstyle{rem}{}{}{}{}{\bf}{.}{ }{}
\theoremstyle{rem}
\newtheorem{rem}[prop]{Remark}
\newtheoremstyle{remm}{}{}{}{}{\bf}{.}{ }{}
\theoremstyle{remm}
\newtheorem{remm}[propp]{Remark}
\newtheorem{df}[prop]{Definition}
\newtheorem{dff}[propp]{Definition}
\begin{document}
\title{Lempert Theorem for strongly linearly convex domains}
\author{\L ukasz Kosi\'nski and Tomasz Warszawski}
\address{Instytut Matematyki, Wydzia\l\ Matematyki i Informatyki, Uniwersytet Jagiello\'nski, ul. Prof. St. \L ojasiewicza 6, 30-348 Krak\'ow, Poland}
\email{lukasz.kosinski@gazeta.pl, tomasz.warszawski@im.uj.edu.pl}
\begin{abstract}
In 1984 L.~Lempert showed that the Lempert function and the Carath\'eodory distance coincide on non-planar bounded strongly linearly convex domains with real analytic boundaries. Following this paper, we present a~slightly modified and more detailed version of the proof. Moreover, the Lempert Theorem is proved for non-planar bounded $\cC^2$-smooth strongly linearly convex domains.
\end{abstract}
\maketitle

The aim of this paper is to present a detailed version of the proof of the Lempert Theorem in the case of non-planar bounded strongly linearly convex domains with smooth boundaries. The original Lempert's proof is presented only in proceedings of a conference (see \cite{Lem1}) with a very limited access and at some places it was quite sketchy. We were encouraged by some colleagues to prepare an extended version of the proof in which all doubts could be removed and some of details of the proofs could be simplified. We hope to have done it below. Certainly, \textbf{the idea of the proof belongs entirely to Lempert}. The main differences, we would like to draw attention to, are
\begin{itemize}

\item results are obtained in $\mathcal C^2$-smooth case;

\item the notion of stationary mappings and $E$-mappings is separated;

\item a geometry of domains is investigated only in neighborhoods of boundaries of stationary mappings (viewed as boundaries of analytic discs) --- this allows us to obtain localization properties for stationary mappings;

\item boundary properties of strongly convex domains are expressed in terms of the squares of their Minkowski functionals.
\end{itemize}

Additional motivation for presenting the proof is the fact, showed recently in \cite{Pfl-Zwo}, that the so-called symmetrized bidisc may be exhausted by strongly linearly convex domains. On the other hand it cannot be exhausted by domains biholomorphic to convex ones (\cite{Edi}). Therefore, the equality of the Lempert function and the Carath\'eodory distance for strongly linearly convex domains does not follow directly from \cite{Lem2}.

\section{Introduction and results}

Let us recall the objects we will deal with. Throughout the paper $\DD$ denotes the unit open disc on the complex plane, $\TT$ is the unit circle and $p$ --- the Poincar\'e distance on $\DD$.

Let $D\subset\CC^{n}$ be a domain and let $z,w\in D$, $v\in\CC^{n}$. The {\it Lempert function}\/ is defined as
\begin{equation}\label{lem}
\widetilde{k}_{D}(z,w):=\inf\{p(0,\xi):\xi\in[0,1)\textnormal{ and }\exists f\in \mathcal{O}(\mathbb{D},D):f(0)=z,\ f(\xi)=w\}.
\end{equation} The {\it Kobayashi-Royden \emph{(}pseudo\emph{)}metric}\/ we define as
\begin{equation}\label{kob-roy}
\kappa_{D}(z;v):=\inf\{\lambda^{-1}:\lambda>0\text{ and }\exists f\in\mathcal{O}(\mathbb{D},D):f(0)=z,\ f'(0)=\lambda v\}.
\end{equation}
Note that
\begin{equation}\label{lem1}
\widetilde{k}_{D}(z,w)=\inf\{p(\zeta,\xi):\zeta,\xi\in\DD\textnormal{ and }\exists f\in \mathcal{O}(\mathbb{D},D):f(\zeta)=z,\ f(\xi)=w\},
\end{equation}
\begin{multline}\label{kob-roy1}
\kappa_{D}(z;v)=\inf\{|\lambda|^{-1}/(1-|\zeta|^2):\lambda\in\CC_*,\,\zeta\in\DD\text{ and }\\ \exists f\in\mathcal{O}(\mathbb{D},D):f(\zeta)=z,\ f'(\zeta)=\lambda v\}.
\end{multline}

If $z\neq w$ (respectively $v\neq 0$), a mapping $f$ for which the infimum in \eqref{lem1} (resp. in \eqref{kob-roy1}) is attained, we call a $\wi{k}_D$-\textit{extremal} (or a \textit{Lempert extremal}) for $z,w$ (resp. a $\kappa_D$-\textit{extremal} for $z,v$). A mapping being a $\wi k_D$-extremal or a $\kappa_D$-extremal we will call just an \textit{extremal} or an \textit{extremal mapping}.

We shall say that $f:\DD\longrightarrow D$ is a unique $\wi{k}_D$-extremal for $z,w$ (resp. a unique $\kappa_D$-extremal for $z,v$) if any other $\wi{k}_D$-extremal $g:\DD\longrightarrow D$ for $z,w$ (resp. $\kappa_D$-extremal for $z,v$) satisfies $g=f\circ a$ for some M\"obius function $a$.

In general, $\wi{k}_{D}$ does not satisfy a triangle inequality --- take for example $D_{\alpha}:=\{(z,w)\in\CC^{2}:|z|,|w|<1,\ |zw|<\alpha\}$, $\alpha\in(0,1)$. Therefore, it is natural to consider the so-called \textit{Kobayashi \emph{(}pseudo\emph{)}distance} given by the formula \begin{multline*}k_{D}(w,z):=\sup\{d_{D}(w,z):(d_{D})\text{ is a family of holomorphically invariant} \\\text{pseudodistances less than or equal to }\widetilde{k}_{D}\}.\end{multline*}
It follows directly from the definition that $$k_{D}(z,w)=\inf\left\{\sum_{j=1}^{N}\wi{k}_{D}(z_{j-1},z_{j}):N\in\NN,\ z_{1},\ldots,z_{N}\in
D,\ z_{0}=z,\ z_{N}=w\right\}.$$

The next objects we are dealing with, are the \textit{Carath\'eodory \emph{(}pseudo\emph{)}distance}
$$c_{D}(z,w):=\sup\{p(F(z),F(w)):F\in\mathcal{O}(D,\DD)\}$$
and the \textit{Carath\'eodory-Reiffen \emph{(}pseudo\emph{)}metric}
$$\gamma_D(z;v):=\sup\{|F'(z)v|:F\in\mathcal{O}(D,\DD),\ F(z)=0\}.$$

A holomorphic mapping $f:\DD\longrightarrow D$ is said to be a \emph{complex geodesic} if $c_D(f(\zeta),f(\xi))=p(\zeta,\xi)$ for any $\zeta,\xi\in\DD$.
\bigskip

Here is some notation. Let $z_1,\ldots,z_n$ be the standard complex coordinates in $\CC^n$ and $x_1,\ldots,x_{2n}$ --- the standard real coordinates in $\CC^n=\RR^n+i\RR^n\simeq\RR^{2n}$. We use $T_{D}^\mathbb{R}(a)$, $T_{D}^\mathbb{C}(a)$ to denote a real and a complex tangent space to a $\cC^1$-smooth domain $D$ at a point $a\in\partial D$, i.e. the sets \begin{align*}T_{D}^\mathbb{R}(a):&=\left\{X\in\CC^{n}:\re\sum_{j=1}^n\frac{\partial r}{\partial z_j}(a)X_{j}=0\right\},\\ T_{D}^\mathbb{C}(a):&=\left\{X\in\CC^{n}:\sum_{j=1}^n\frac{\partial r}{\partial z_j}(a)X_{j}=0\right\},\end{align*}
where $r$ is a defining function of $D$. Let $\nu_D(a)$ be the outward unit normal vector to $\partial D$ at $a$.

Let $\mathcal{C}^{k}(\CDD)$, where $k\in(0,\infty]$, denote a class of continuous functions on $\CDD$, which are of class $\cC^k$ on $\DD$ and
\begin{itemize}
\item if $k\in\NN\cup\{\infty\}$ then derivatives up to the order $k$ extend continuously on~$\CDD$;
\item if $k-[k]=:c>0$ then derivatives up to the order $[k]$ are $c$-H\"older continuous on $\DD$.
\end{itemize}
By $\mathcal{C}^\omega$ class we shall denote real analytic functions. Further, saying that $f$ is of class $\mathcal{C}^{k}(\TT)$, $k\in(0,\infty]\cup\{\omega\}$, we mean that the function $t\longmapsto f(e^{it})$, $t\in\RR$, is in $\mathcal{C}^{k}(\mathbb R)$. For a compact set $K\su\CC^n$ let $\OO(K)$ denote the set of functions extending holomorphically on a neighborhood of $K$ (we assume that all neighborhoods are open). In that case we shall sometimes say that a given function is of class $\OO(K)$. Note that $\CLW(\TT)=\OO(\TT)$. 

Let $|\cdot|$ denote the Euclidean norm in $\CC^{n}$ and let $\dist(z,S):=\inf\{|z-s|:s\in S\}$ be a distance of the point $z\in\CC^n$ to the set $S\su\CC^n$. For such a set $S$ we define $S_*:=S\setminus\{0\}$. Let $\BB_n:=\{z\in\CC^n:|z|=1\}$ be the unit ball and $B_n(a,r):=\{z\in\CC^n:|z-a|<r\}$ --- an open ball with a center $a\in\CC^n$ and a radius $r>0$. Put $$z\bullet w:=\sum_{j=1}^nz_{j}{w}_{j}$$ for $z,w\in\CC^{n}$ and let $\langle\cdotp,-\rangle$ be a hermitian inner product on $\CC^n$. The real inner product on $\CC^n$ is denoted by $\langle\cdotp,-\rangle_{\RR}=\re\langle\cdotp,-\rangle$.

We use $\nabla$ to denote the gradient $(\pa/\pa x_1,\ldots,\pa/\pa x_{2n})$. For real-valued functions the gradient is naturally identified with $2(\pa/\pa\ov z_1,\ldots,\pa/\pa\ov z_n)$. Recall that $$\nu_D(a)=\frac{\nabla r(a)}{|\nabla r(a)|}.$$ Let $\mathcal{H}$ be the Hessian matrix $$\left[\frac{\pa^2}{\pa x_j\pa x_k}\right]_{1\leq j,k\leq 2n}.$$ Sometimes, for a $\cC^2$-smooth function $u$ and a vector $X\in\RR^{2n}$ the Hessian $$\sum_{j,k=1}^{2n}\frac{\partial^2 u}{\partial x_j\partial x_k}(a)X_{j}X_{k}=X^T\HH u(a)X$$ will be denoted by $\HH u(a;X)$. By $\|\cdot\|$ we denote the operator norm.
\bigskip
\begin{df}\label{29}
Let $D\subset\CC^{n}$ be a domain.

We say that $D$ is \emph{linearly convex} (resp. \emph{weakly linearly convex}) if through any point $a\in\mathbb C^n\setminus D$ (resp. $a\in \partial D$) there goes an $(n-1)$-dimensional complex hyperplane disjoint from $D$.

A domain $D$ is said to be \emph{strongly linearly convex} if
\begin{enumerate}
\item $D$ has $\mathcal{C}^{2}$-smooth boundary;
\item there exists a defining function $r$ of $D$ such that
\begin{equation}\label{48}\sum_{j,k=1}^n\frac{\partial^2 r}{\partial z_j\partial\overline z_k}(a)X_{j}\overline{X}_{k}>\left|\sum_{j,k=1}^n\frac{\partial^2 r}{\partial z_j\partial z_k}(a)X_{j}X_{k}\right|,\ a\in\partial D,\ X\in T_{D}^\mathbb{C}(a)_*.\end{equation}
\end{enumerate}

More generally, any point $a\in\pa D$ for which there exists a defining function $r$ satisfying \eqref{48}, is called a \emph{point of the strong linear convexity} of $D$.

Furthermore, we say that a domain $D$ has \emph{real analytic boundary} if it possesses a real analytic defining function.
\end{df}

Note that the condition \eqref{48} does not depend on the choice of a defining function of $D$.

\begin{rem}
Let $D\subset\CC^{n}$ be a strongly linearly convex domain. Then
\begin{enumerate}
\item any $(n-1)$-dimensional complex tangent hyperplane intersects $\partial{D}$ at precisely one point; in other words $$\overline D\cap(a+T_{D}^\mathbb{C}(a))=\{a\},\ a\in\pa D;$$
\item for $a\in\pa D$ the equation $\langle w-a, \nu_D(a)\rangle=0$ describes the $(n-1)$-dimensional complex tangent hyperplane $a+T_{D}^\mathbb{C}(a)$, consequently $$\langle z-a, \nu_D(a)\rangle\neq 0,\ z\in D,\ a\in\pa D.$$
\end{enumerate}
\end{rem}
\bigskip
The main aim of the paper is to present a detailed proof of the following

\begin{tw}[Lempert Theorem]\label{lem-car}
Let $D\subset\CC^{n}$, $n\geq 2$, be a bounded strongly linearly convex domain. Then $$c_{D}=k_{D}=\wi{k}_{D}\text{\,\ and\,\, }\gamma_D=\kappa_D.$$
\end{tw}

\bigskip

An important role will be played by strongly convex domains and strongly convex functions.
\begin{df}
A domain $D\subset\CC^{n}$ is called \emph{strongly convex} if
\begin{enumerate}
\item $D$ has $\mathcal{C}^{2}$-smooth boundary;
\item there exists a defining function $r$ of $D$ such that
\begin{equation}\label{sc}\sum_{j,k=1}^{2n}\frac{\partial^2 r}{\partial x_j\partial x_k}(a)X_{j}X_{k}>0,\ a\in\partial D,\ X\in T_{D}^\mathbb{R}(a)_*.\end{equation}
\end{enumerate}
Generally, any point $a\in\pa D$ for which there exists a defining function $r$ satisfying \eqref{sc}, is called a \emph{point of the strong convexity} of $D$.
\end{df}
\begin{rem}
A strongly convex domain $D\subset\CC^{n}$ is convex and strongly linearly convex. Moreover, it is strictly convex, i.e. for any different points $a,b\in\overline D$ the interior of the segment $[a,b]=\{ta+(1-t)b:t\in [0,1]\}$ is contained in $D$ (i.e. $ta+(1-t)b\in D$ for any $t\in(0,1)$).

Observe also that any bounded convex domain with a real analytic boundary is strictly convex. Actually, if a domain $D$ with a real analytic boundary were not strictly convex, then we would be able to find two distinct points $a,b\in\pa D$ such that the segment $[a,b]$ lies entirely in $\partial D$. On the other hand, the identity principle would imply that the set $\{t\in\mathbb R:\exists\eps>0:sa+(1-s)b\in\pa D\text{ for }|s-t|<\eps\}$ is open-closed in $\mathbb R$. Therefore it has to be empty. This immediately gives a contradiction.
\end{rem}

\begin{rem}
It is well-known that for any convex domain $D\su\CC^{n}$ there is a sequence $\{D_m\}$ of bounded strongly convex domains with real analytic boundaries, such that $D_m\su D_{m+1}$ and $\bigcup_m D_m=D$. 

In particular, Theorem~\ref{lem-car} holds for convex domains.
\end{rem}

\begin{df}
Let $U\su\CC^n$ be a domain. A function $u:U\longrightarrow\RR$ is called \emph{strongly convex} if
\begin{enumerate}
\item $u$ is $\mathcal{C}^{2}$-smooth;
\item $$\sum_{j,k=1}^{2n}\frac{\partial^2 u}{\partial x_j\partial x_k}(a)X_{j}X_{k}>0,\ a\in U,\ X\in(\RR^{2n})_*.$$
\end{enumerate}
\end{df}

\begin{df} A degree of a continuous function (treated as a curve) $:\mathbb T\longrightarrow\mathbb T$ is called its winding number. The fundamental group is a homotopy invariant. Thus the definition of the \emph{winding number of a continuous function} $\phi:\mathbb T\longrightarrow\mathbb C_*$ is the same. We denote it by $\wind\phi$. 

In the case of a $\cC^1$-smooth function $\phi:\TT\longrightarrow\CC_*$, its winding number is just the index of $\phi$ at 0, i.e. $$\wind\phi=\frac{1}{2\pi i}\int_{\phi(\TT)}\frac{d\zeta}{\zeta}=\frac{1}{2\pi i}\int_{0}^{2\pi}\frac{\frac{d}{dt}\phi(e^{it})}{\phi(e^{it})}dt.$$
\end{df}

\begin{rem}\label{49}
\begin{enumerate}
\item\label{51} If $\phi\in\cC(\TT,\CC_*)$ extends to a function $\widetilde{\phi}\in\OO(\DD)\cap \mathcal C(\CDD)$ then $\wind\phi$ is the number of zeroes of $\widetilde{\phi}$ in $\DD$ counted with multiplicities;
\item\label{52} $\wind(\phi\psi)=\wind\phi+\wind\psi$, $\phi,\psi\in\cC(\TT,\CC_*)$;
\item\label{53} $\wind\phi=0$ if $\phi\in\cC(\TT)$ and $\re\phi>0$.
\end{enumerate}
\end{rem}

\begin{df}
The boundary of a domain $D$ of $\mathbb C^n$ is \emph{real analytic in a neighborhood} $U$ of the set $S\su\pa D$ if there exists a function $r\in\mathcal C^{\omega}(U,\RR)$ such that $D\cap U=\{z\in U:r(z)<0\}$ and $\nabla r$ does not vanish in $U$.
\end{df}

\begin{df}\label{21}
Let $D\subset\CC^{n}$ be a domain. We call a holomorphic mapping $f:\DD\longrightarrow D$ a \emph{stationary mapping} if
\begin{enumerate}
\item $f$ extends to a holomorphic mapping in a neighborhood od $\CDD$ $($denoted by the same letter$)$;
\item $f(\TT)\subset\partial D$;
\item there exists a real analytic function
$\rho:\TT\longrightarrow\RR_{>0}$ such that the mapping $\TT\ni\zeta\longmapsto\zeta
\rho(\zeta)\overline{\nu_D(f(\zeta))}\in\CC^{n}$ extends to a mapping holomorphic in a neighborhood of $\CDD$ $($denoted by $\widetilde{f}${$)$}.
\end{enumerate}

Furthermore, we call a holomorphic mapping $f:\DD\longrightarrow D$ a \emph{weak stationary mapping} if
\begin{enumerate}
\item[(1')] $f$ extends to a $\cC^{1/2}$-smooth mapping on $\CDD$ $($denoted by the same letter$)$;
\item[(2')] $f(\TT)\subset\partial D$;
\item[(3')] there exists a $\cC^{1/2}$-smooth function
$\rho:\TT\longrightarrow\RR_{>0}$ such that the mapping $\TT\ni\zeta\longmapsto\zeta
\rho(\zeta)\overline{\nu_D(f(\zeta))}\in\CC^{n}$ extends to a mapping $\widetilde{f}\in\OO(\DD)\cap\cC^{1/2}(\CDD)$.
\end{enumerate}

The definition of a $($weak$)$ stationary mapping $f:\mathbb D\longrightarrow D$ extends naturally to the case when $\pa D$ is real analytic in a neighborhood of $f(\TT)$.
\end{df}

Directly from the definition of a stationary mapping $f$, it follows that $f$ and $\wi f$ extend holomorphically on some neighborhoods of $\CDD$. By $\DD_f$ we shall denote their intersection.

\begin{df}\label{21e}
Let $D\su\CC^n$, $n\geq 2$, be a bounded strongly linearly convex domain with real analytic boundary. A holomorphic mapping $f:\DD\longrightarrow D$ is called a (\emph{weak}) $E$-\emph{mapping} if it is a (weak) stationary mapping and
\begin{enumerate}
\item[(4)] setting $\varphi_z(\zeta):=\langle z-f(\zeta),\nu_D(f(\zeta))\rangle,\ \zeta\in\TT$, we have $\wind\phi_z=0$ for some $z\in D$.
\end{enumerate}
\end{df}

\begin{rem}
The strong linear convexity of $D$ implies $\varphi_z(\zeta)\neq 0$ for any $z\in D$ and $\zeta\in\TT$. Therefore, $\wind\phi_z$ vanishes for all $z\in D$ if it vanishes for some $z\in D$.

Additionally, any stationary mapping of a convex domain is an $E$-mapping (as $\re \varphi_z<0$).
\end{rem}

We shall prove that in a class of non-planar bounded strongly linearly convex domains with real analytic boundaries weak stationary mappings are just stationary mappings, so there is no difference between $E$-mappings and weak $E$-mappings. 

We have the following result describing extremal mappings, which is very interesting in its own.

\begin{tw}\label{main} Let $D\su\CC^n$, $n\geq 2$, be a bounded strongly linearly convex domain. 

Then a holomorphic mapping $f:\DD\longrightarrow D$ is an extremal if and only if $f$ is a weak $E$-mapping.

For a domain $D$ with real analytic boundary, a holomorphic mapping $f:\mathbb D\longrightarrow D$ is an extremal if and only if $f$ is an $E$-mapping.

If $\pa D$ is of class $\cC^k$, $k=3,4,\ldots,\infty$, then any weak $E$-mapping $f:\DD\longrightarrow D$ and its associated mappings $\wi f,\rho$ are $\mathcal C^{k-1-\eps}$-smooth for any $\eps>0$.

\end{tw}

The idea of the proof of the Lempert Theorem is as follows. In real analytic case we shall show that $E$-mappings are complex geodesics (because they have left inverses). Then we shall prove that for any different points $z,w\in D$ (resp. for a point $z\in D$ and a vector $v\in(\CC^n)_*$) there is an $E$-mapping passing through $z,w$ (resp. such that $f(0)=z$ and $f'(0)=v$). This will give the equality between the Lempert function and the Carath\'eodory distance. In the general case, we exhaust a $\cC^2$-smooth domain by strongly linearly convex domains with real analytic boundaries.

To prove Theorem \ref{main} we shall additionally observe that (weak) $E$-mappings are unique extremals.
\bigskip

\begin{center}{\sc Real analytic case}\end{center}
\bigskip

In what follows and if not mentioned otherwise, $D\su\CC^n$, $n\geq 2$, is a \textbf{bounded strongly linearly convex domain with real analytic boundary}.
\section{Weak stationary mappings of strongly linearly convex domains with real analytic boundaries are stationary mappings}\label{55}
Let $M\subset\CC^m$ be a totally real $\CLW$ submanifold of the real dimension $m$. Fix a point $z\in M$. There are neighborhoods $U,V\su\CC^m$ of $0$ and $z$ respectively and a biholomorphic mapping $\Phi:U\longrightarrow V$ such that $\Phi(\RR^m\cap U)=M\cap V$ (for the proof see Appendix).

\begin{prop}\label{6}
A weak stationary mapping of $D$ is a stationary mapping of $D$ with the same associated mappings.
\end{prop}
\begin{proof}
Let $f:\DD\longrightarrow D$ be a weak stationary mapping. Our aim is to prove that $f,\widetilde{f}\in\OO(\CDD)$ and $\rho\in\mathcal C^{\omega}(\TT)$. Choose a point $\zeta_0\in\TT$. Since $\widetilde{f}(\zeta_0)\neq 0$, we can assume that $\widetilde{f}_1(\zeta)\neq 0$ in $\CDD\cap U_0$, where $U_0$ is a neighborhood of $\zeta_0$. This implies
$\nu_{D,1}(f(\zeta_0))\neq 0$, so $\nu_{D,1}$ does not vanish on some set $V_0\su\pa D$, relatively open in
$\pa D$, containing the point $f(\zeta_0)$. Shrinking $U_0$, if necessary, we may assume that $f(\TT\cap U_0)\subset V_0$.

Define $\psi:V_0\longrightarrow\CC^{2n-1}$ by
$$\psi(z)=\left(z_1,\ldots,z_n,
\ov{\left(\frac{\nu_{D,2}(z)}{\nu_{D,1}(z)}\right)},\ldots,\ov{\left(\frac{\nu_{D,n}(z)}{\nu_{D,1}(z)}\right)}\right).$$ The set $M:=\psi(V_0)$ is the graph of a $\CLW$ function defined on the local $\CLW$ submanifold $V_0$, so it is a local $\CLW$ submanifold in $\CC^{2n-1}$ of the real dimension $2n-1$. Assume for a moment that $M$ is totally real.

Let $$g(\zeta):=\left(f_1(\zeta),\ldots,f_n(\zeta),
\frac{\widetilde{f}_2(\zeta)}{\widetilde{f}_1(\zeta)},\ldots,\frac{\widetilde{f}_n(\zeta)}{\widetilde{f}_1(\zeta)}\right),\ \zeta\in\CDD\cap U_0.$$ If $\zeta\in\TT\cap U_0$ then
$\widetilde{f}_k(\zeta)\widetilde{f}_1(\zeta)^{-1} =
\overline{\nu_{D,k}(f(\zeta))}\ \overline{\nu_{D,1}(f(\zeta))}^{-1}$, so
$g(\zeta)=\psi(f(\zeta))$. Therefore, $g(\TT\cap U_0)\subset M$. Thanks to the Reflection
Principle (see Appendix), $g$ extends holomorphically past $\TT\cap U_0$, so $f$ extends holomorphically on a neighborhood of $\zeta_0$.

The mapping $\overline{\nu_D\circ f}$ is real analytic on $\TT$, so it extends to a mapping $h$ holomorphic in a neighborhood $W$ of $\TT$. For $\zeta\in\TT\cap U_0$ we have $$\frac{\zeta
h_1(\zeta)}{\widetilde{f}_1(\zeta)}=\frac{1}{\rho(\zeta)}.$$ The function on the
left side is holomorphic in $\DD\cap U_0\cap W$ and continuous in $\CDD\cap U_0\cap W$. Since it
has real values on $\TT\cap U_0$, the Reflection Principle implies that it is holomorphic in a neighborhood of $\TT\cap U_0$. Hence $\rho$ and $\widetilde{f}$ are holomorphic in a neighborhood of $\zeta_0$. Since $\zeta_0$ is arbitrary, we get the assertion.

It remains to prove that $M$ is totally real. Let $r$ be a defining function of $D$. Recall that for any point $z\in V_0$ $$\frac{\ov{\nu_{D,k}(z)}}{\ov{\nu_{D,1}(z)}}=\frac{\partial r}{\partial z_k}(z)\left(\frac{\partial r}{\partial z_1}(z)\right)^{-1},\,k=1,\ldots,n.$$
Consider the mapping $S=(S_1,\ldots,S_n):V_0\times\CC^{n-1}\longrightarrow\RR\times\CC^{n-1}$
given by $$S(z,w):=\left(r(z),\frac{\partial r}{\partial z_2}(z)-w_{1}\frac{\partial r}{\partial z_1}(z),\ldots,\frac{\partial r}{\partial z_n}(z)-w_{n-1}\frac{\partial r}{\partial z_1}(z)\right).$$ Clearly, $M=S^{-1}(\{0\})$. Hence
\begin{equation}\label{tan} T_{M}^{\RR}(z,w)\subset\ker\nabla S(z,w),\ (z,w)\in M,\end{equation} where
$\nabla S:=(\nabla S_1,\ldots,\nabla S_n)$.

Fix a point $(z,w)\in M$. Our goal is to prove that $T_{M}^{\CC}(z,w)=\lbrace 0\rbrace$. Take an arbitrary vector $(X,Y)=(X_1,\ldots,X_n,Y_1,\ldots,Y_{n-1})\in T_{M}^{\CC}(z,w)$. Then we infer from \eqref{tan} that $$\sum_{k=1}^n\frac{\partial r}{\partial z_k}(z)X_k=0,$$ i.e. $X\in T_{D}^{\CC}(z)$. Denoting $v:=(z,w)$, $V:=(X,Y)$ and making use of \eqref{tan} again we find that
$$0=\nabla S_k(v)(V)=\sum_{j=1}^{2n-1}\frac{\pa S_k}{\pa v_j}(v)V_j+\sum_{j=1}^{2n-1}\frac{\pa S_k}{\pa\ov v_j}(v)\ov V_j$$ for $k=2,\ldots,n$.
But $V\in T_{M}^{\CC}(v)$, so $iV\in T_{M}^{\CC}(v)$. Thus $$0=\nabla S_k(v)(iV)=i\sum_{j=1}^{2n-1}\frac{\pa S_k}{\pa v_j}(v)V_j-i\sum_{j=1}^{2n-1}\frac{\pa S_k}{\pa\ov v_j}(v)\ov V_j.$$ In particular, \begin{multline*}0=\sum_{j=1}^{2n-1}\frac{\pa S_k}{\pa\ov v_j}(v)\ov V_j=\sum_{j=1}^{n}\frac{\pa S_k}{\pa\ov z_j}(z,w)\ov X_j+\sum_{j=1}^{n-1}\frac{\pa S_k}{\pa\ov w_j}(z,w)\ov Y_j=\\=\sum_{j=1}^n\frac{\partial^2r}{\partial z_k\partial\overline{z}_j}(z)\overline X_j-w_{k-1}\sum_{j=1}^n\frac{\partial^2r}{\partial z_1\partial\overline{z}_j}(z)\overline X_j.
\end{multline*}
The equality $M=S^{-1}(\{0\})$ gives $$w_{k-1}=\frac{\partial r}{\partial z_k}(z)\left(\frac{\partial r}{\partial z_1}(z)\right)^{-1},$$ so $$\frac{\partial r}{\partial z_1}(z)\sum_{j=1}^n\frac{\partial^2r}{\partial z_k\partial\overline{z}_j}(z)\overline X_j=\frac{\partial r}{\partial z_k}(z)\sum_{j=1}^n\frac{\partial^2r}{\partial z_1\partial\overline{z}_j}(z)\overline X_j,\ k=2,\ldots,n.$$ Note that the last equality holds also for $k=1$. Therefore, \begin{multline*}
\frac{\partial r}{\partial z_1}(z)\sum_{j,k=1}^n\frac{\partial^2r}{\partial z_k\partial\overline{z}_j}(z)\overline X_jX_k=\sum_{k=1}^n\frac{\partial r}{\partial z_k}(z)\sum_{j=1}^n\frac{\partial^2r}{\partial z_1\partial\overline{z}_j}(z)\overline X_jX_k =\\=\left(\sum_{k=1}^n\frac{\partial r}{\partial z_k}(z)X_k\right)\left(\sum_{j=1}^n\frac{\partial^2r}{\partial z_1\partial\overline{z}_j}(z)\overline X_j\right)=0.
\end{multline*}
By the strong linear convexity of $D$ we have $X=0$. This implies $Y=0$, since $$0=\nabla S_k(z,w)(0,Y)=\sum_{j=1}^{n-1}\frac{\pa S_k}{\pa w_j}(v)Y_j+\sum_{j=1}^{n-1}\frac{\pa S_k}{\pa\ov w_j}(v)\ov Y_j=-\frac{\partial r}{\partial z_1}(z)Y_{k-1}$$ for $k=2,\ldots,n$. 
\end{proof}

\section{(Weak) $E$-mappings vs. extremal mappings and complex geodesics}

In this section we will prove important properties of (weak) $E$-mappings. In particular, we will show that they are complex geodesics and unique extremals.
\subsection{Weak $E$-mappings are complex geodesics and unique extremals}
The results of this subsection are related to weak $E$-mappings of bounded strongly linearly convex domains $D\su\CC^n$, $n\geq 2$.

Let $$G(z,\zeta):=(z-f(\zeta))\bullet\widetilde{f}(\zeta),\ z\in\CC^n,\ \zeta\in\DD_f.$$

\begin{propp}\label{1}
Let $D\su\CC^n$, $n\geq 2$, be a bounded strongly linearly convex domain and let $f:\DD\longrightarrow D$ be a weak $E$-mapping. Then there exist an open set $W\supset\overline D\setminus f(\TT)$ and a holomorphic mapping $F:W\longrightarrow\DD$ such that for any $z\in W$ the number $F(z)$ is a unique solution of the equation $G(z,\zeta)=0,\ \zeta\in\DD$. In particular, $F\circ f=\id_{\DD}$.
\end{propp}

In the sequel we will strengthen the above proposition for domains with real analytic boundaries (see Proposition~\ref{34}).

\begin{proof}[Proof of Proposition~\ref{1}]
Set $A:=\overline{D}\setminus f(\TT)$. Since $D$ is strongly linearly convex, $\varphi_z$ does not vanish in $\TT$ for any $z\in A$, so by a continuity argument the condition (4) of Definition~\ref{21e} holds for every $z$ in some open set $W\supset A$. For a fixed $z\in W$ we have $$G(z,\zeta)=\zeta\rho(\zeta)\varphi_z(\zeta),\ \zeta\in\TT,$$ so $\wind G(z,\cdotp)=1$. Since $G(z,\cdotp)\in\OO(\DD)$, it has in $\DD$ exactly one simple root $F(z)$. Hence $G(z,F(z))=0$ and $\frac{\partial G}{\partial\zeta}(z,F(z))\neq 0$. By the Implicit Function Theorem, $F$ is holomorphic in $W$. The equality $F(f(\zeta))=\zeta$ for $\zeta\in\DD$ is clear.
\end{proof}

From the proposition above we immediately get the following
\begin{corr}\label{5}
A weak $E$-mapping $f:\DD\longrightarrow D$ of a bounded strongly linearly convex domain $D\su\CC^n$, $n\geq 2$, is a complex geodesic. In particular,
$$c_{D}(f(\zeta),f(\xi))=\wi k_D(f(\zeta),f(\xi))\text{\,\ and\,\, }\gamma_D(f(\zeta);f'(\zeta))=\kappa_D(f(\zeta);f'(\zeta)),$$ for any $\zeta,\xi\in\DD$.
\end{corr}

Using left inverses of weak $E$-mappings we may prove the uniqueness of extremals.
\begin{propp}\label{2}
Let $D\su\CC^n$, $n\geq 2$, be a bounded strongly linearly convex domain and let $f:\DD\longrightarrow D$ be a weak $E$-mapping. Then for any $\xi\in(0,1)$ the mapping $f$ is a unique $\wi{k}_D$-extremal for $z=f(0)$, $w=f(\xi)$ \emph{(}resp. a unique $\kappa_D$-extremal for $z=f(0)$, $v=f'(0)$\emph{)}.
\end{propp}
\begin{proof}

Suppose that $g$ is a $\wi{k}_D$-extremal for $z,w$ (resp. a $\kappa_D$-extremal for $z,v$) such that $g(0)=z$, $g(\xi)=w$ (resp. $g(0)=z$, $g'(0)=v$). Our aim is to show that $f=g$. Proposition~\ref{1} provides us with the mapping $F$, which is a left inverse for $f$. By the Schwarz Lemma, $F$ is a left inverse for $g$, as well, that is $F\circ g=\text{id}_{\DD}$. We claim that $\lim_{\DD\ni\zeta\to\zeta_0}g(\zeta)=f(\zeta_0)$ for any $\zeta_0\in\TT$ (in particular, we shall show that the limit does exist).

Assume the contrary. Then there are $\zeta_0\in\TT$ and a sequence $\{\zeta_m\}\subset\DD$ convergent to $\zeta_0$ such that the limit $Z:=\lim_{m\to\infty}g(\zeta_m)\in\overline{D}$ exists and is not equal to $f(\zeta_0)$. We have $G(z,F(z))=0$, so putting $z=g(\zeta_m)$ we infer that $$0=(g(\zeta_m)-f(F(g(\zeta_m))))\bullet \widetilde{f}(F(g(\zeta_m)))=(g(\zeta_m)-f(\zeta_m))\bullet\widetilde{f}(\zeta_m).
$$ Passing with $m$ to the infinity we get $$0=(Z-f(\zeta_0))\bullet \widetilde{f}(\zeta_0)=\zeta_0\rho(\zeta_0)\langle Z-f(\zeta_0),\nu_D(f(\zeta_0))\rangle.$$ This means that $Z-f(\zeta_0)\in T^{\CC}_D(f(\zeta_0))$. Since $D$ is strongly linearly convex, we deduce that $Z=f(\zeta_0)$, which is a contradiction.

Hence $g$ extends continuously on $\CDD$ and, by the maximum principle, $g=f$.
\end{proof}

\begin{propp}\label{3}
Let $D\su\CC^n$, $n\geq 2$, be a bounded strongly linearly convex domain, let $f:\DD\longrightarrow D$ be a weak $E$-mapping and let $a$ be an automorphism of $\DD$. Then $f\circ a$ is a weak $E$-mapping of $D$.
\end{propp}
\begin{proof}
Set $g:=f\circ a$.

Clearly, the conditions (1') and (2') of Definition~\ref{21} are satisfied by $g$.

To prove that $g$ satisfies the condition (4) of Definition~\ref{21e} fix a point $z\in D$. Let $\varphi_{z,f}$, $\varphi_{z,g}$ be the functions appearing in the condition (4) for $f$ and $g$ respectively. Then $\varphi_{z,g}=\varphi_{z,f}\circ a$. Since $a$ maps $\TT$ to $\TT$ diffeomorphically, we have $\wind\varphi_{z,g}=\pm\wind\varphi_{z,f}=0$.

It remains to show that the condition (3') of Definition~\ref{21} is also satisfied by $g$. Note that the function $\wi a(\zeta):=\zeta/a(\zeta)$ has a holomorphic branch of the logarithm in the neighborhood of $\TT$. This follows from the fact that $\wind \wi a=0$, however the existence of the holomorphic branch may be shown quite elementary. Actually, it would suffices to prove that $\wi a(\TT)\neq\TT$. Expand $a$ as $$a(\zeta)=e^{it}\frac{\zeta-b}{1-\overline b\zeta}$$ with some $t\in\RR$, $b\in\DD$ and observe that $\widetilde a$ does not attain the value $-e^{-it}$. Indeed, if $\zeta/a(\zeta)=-e^{-it}$ for some $\zeta\in\TT$, then $$\frac{1-\overline b\zeta}{1-b\overline\zeta}=-1,$$ so $2=2\re(b\overline\zeta)\leq 2|b|$, which is impossible.

Concluding, there exists a function $v$ holomorphic in a neighborhood of $\TT$ such that $$\frac{\zeta}{a(\zeta)}=e^{i v(\zeta)}.$$ Note that $v(\TT)\su\RR$. Expanding $v$ in Laurent series $$v(\zeta)=\sum_{k=-\infty}^{\infty}a_k\zeta^k,\ \zeta\text{ near }\TT,$$ we infer that $a_{-k}=\overline a_k$, $k\in\ZZ$. Therefore, $$v(\zeta)=a_0+\sum_{k=1}^\infty 2\re(a_k\zeta^k)=\re\left(a_0+2\sum_{k=1}^\infty a_k\zeta^k\right),\ \zeta\in\TT.$$ Hence, there is a function $h$ holomorphic in the neighborhood of $\CDD$ such that $v=\im h$. Put $u:=h-iv$. Then $u\in\OO(\TT)$ and $u(\TT)\su\RR$.

Take $\rho$ be as in the condition (3') of Definition~\ref{21} for $f$ and define $$r(\zeta):=\rho(a(\zeta))e^{u(\zeta)},\ \zeta\in\TT.$$ Let us compute
\begin{eqnarray*}\zeta r(\zeta)\overline{\nu_D(g(\zeta))}=\zeta u^{u(\zeta)}\rho(a(\zeta))\overline{\nu_D(f(a(\zeta)))}&=&\\=a(\zeta)h(\zeta)\rho(a(\zeta))\overline{\nu_D(f(a(\zeta)))}
&=&h(\zeta)\widetilde{f}(a(\zeta)),\quad\zeta\in\TT.
\end{eqnarray*} Thus $\zeta\longmapsto\zeta r(\zeta)\overline{\nu_D(g(\zeta))}$ extends holomorphically to a function of class $\OO(\DD)\cap\cC^{1/2}(\CDD)$.
\end{proof}

\begin{corr}\label{28}
A weak $E$-mapping $f:\DD\longrightarrow D$ of a bounded strongly linearly convex domain $D\su\CC^n$, $n\geq 2$, is a unique $\wi{k}_D$-extremal for $f(\zeta),f(\xi)$ \emph{(}resp. a unique $\kappa_D$-extremal for $f(\zeta),f'(\zeta)$\emph{)}, where $\zeta,\xi\in\DD$, $\zeta\neq\xi$.
\end{corr}

\subsection{Generalization of Proposition~\ref{1}}
The results obtained in this subsection will play an important role in the sequel.

We start with 
\begin{propp}\label{4}
Let $f:\DD\longrightarrow D$ be an $E$-mapping. Then the function $f'\bullet\widetilde{f}$ is a positive constant.
\end{propp}
\begin{proof}
Consider the curve $$\RR\ni t\longmapsto f(e^{it})\in\partial D.$$ Its any tangent vector $ie^{it}f'(e^{it})$ belongs to $T_{D}^\mathbb{R}(f(e^{it}))$, i.e. $$\re\langle ie^{it}f'(e^{it}),\nu_D(f(e^{it}))\rangle=0.$$ Thus for $\zeta\in\TT$ $$0=\rho(\zeta)\re\langle i\zeta f'(\zeta),\nu_D(f(\zeta))\rangle=-\im f'(\zeta)\bullet\widetilde{f}(\zeta),$$ so the holomorphic function $f'\bullet\widetilde{f}$ is a real constant $C$.

Considering the curve $$[0,1+\eps)\ni t\longmapsto f(t)\in\overline D$$ for small $\eps>0$ and noting that $f([0,1))\su D$, $f(1)\in\partial D$, we see that the derivative of $r\circ f$ at a point $t=1$ is non-negative, where $r$ is a defining function of $D$. Hence $$0\leq\re\langle f'(1),\nu_D(f(1))\rangle =\frac{1}{\rho(1)} \re( f'(1)\bullet\widetilde{f}(1))=
\frac{C}{\rho(1)},$$ i.e. $C\geq 0$. For $\zeta\in\TT$
$$\frac{f(\zeta)-f(0)}{\zeta}\bullet\widetilde{f}(\zeta)=\rho(\zeta)\langle f (\zeta)-f(0),\nu_D(f(\zeta))\rangle.$$ This function has the winding number equal to $0$. Therefore, the function $$g(\zeta):=\frac{f(\zeta)-f(0)}{\zeta}\bullet\widetilde{f}(\zeta),$$ which is holomorphic in a neighborhood of $\CDD$, does not vanish
in $\DD$. In particular, $C=g(0)\neq 0$.
\end{proof}
The function $\rho$ is defined up to a constant factor. \textbf{We choose $\rho$ so that $ f'\bullet\widetilde{f}\equiv 1$}, i.e. \begin{equation}\label{rho}\rho(\zeta)^{-1}=\langle\zeta f'(\zeta),\nu_D(f(\zeta))\rangle,\ \zeta\in\TT.\end{equation} In that way $\widetilde{f}$ and $\rho$ are uniquely determined by $f$.

\begin{propp}
An $E$-mapping $f:\DD\longrightarrow D$ is injective in $\CDD$.
\end{propp}

\begin{proof}The function $f$ has the left-inverse in $\DD$, so it suffices to check the injectivity on $\TT$. Suppose that $f(\zeta_1)=f(\zeta_2)$ for some $\zeta_1,\zeta_2\in\TT$, $\zeta_1\neq\zeta_2$, and consider the curves $$\gamma_j:[0,1]\ni t\longmapsto f(t\zeta_j)\in\overline D,\ j=1,2.$$ Since $$\re\langle\gamma_j'(1),\nu_D(f(\zeta_j))\rangle=\re\langle\zeta_jf'(\zeta_j),\nu_D(f(\zeta_j))\rangle
=\rho(\zeta_j)^{-1}\neq 0,$$ the curves $\gamma_j$ hit $\pa D$ transversally at their common point $f(\zeta_1)$. We claim that there exists $C>0$ such that for $t\in(0,1)$ close to $1$ there is $s_t\in(0,1)$ satisfying $\wi k_D(f(t\zeta_1),f(s_t\zeta_2))<C$. It will finish the proof since $$\wi k_D(f(t\zeta_1),f(s_t\zeta_2))=p(t\zeta_1,s_t\zeta_2)\to\infty,\ t\to 1.$$ We may assume that $f(\zeta_1)=0$ and $\nu_D(0)=(1,0,\ldots,0)=:e_1$. There exists a ball $B\su D$ tangent to $\pa D$ at $0$. Using a homothety, if necessary, one can assume that $B=\BB_n-e_1$. From the transversality of $\gamma_1,\gamma_2$ to $\pa D$ there exists a cone $$A:=\{z\in\CC^n:-\re z_1>k|z|\},\quad k>0,$$ such that $\gamma_1(t),\gamma_2(t)\in A\cap B$ if $t\in(0,1)$ is close to $1$. For $z\in A$ let $k_z>k$ be a positive number satisfying the equality $$|z|=\frac{-\re z_1}{k_z}.$$ 

Note that for any $a\in\gamma_1((0,1))$ sufficiently close to $0$ one may find $b\in\gamma_2((0,1))\cap A\cap B$ such that $\re b_1=\re a_1$. To get a contradiction it suffices to show that $\wi k_D(a,b)$ is bounded from above by a constant independent on $a$ and $b$. 

We have the following estimate \begin{multline*}\wi k_D(a,b)\leq\wi k_{\BB_n-e_1}(a,b)=\wi k_{\BB_n}(a+e_1,b+e_1)=\\=\tanh^{-1}\sqrt{1-\frac{(1-|a+e_1|^2)(1-|b+e_1|^2)}{|1-\langle a+e_1,b+e_1 \rangle|^2}}.\end{multline*} The last expression is bounded from above if and only if $$\frac{(1-|a+e_1|^2)(1-|b+e_1|^2)}{|1-\langle a+e_1,b+e_1\rangle|^2}$$ is bounded from below by some positive constant. We estimate $$\frac{(1-|a+e_1|^2)(1-|b+e_1|^2)}{|1-\langle a+e_1,b+e_1\rangle|^2}=\frac{(2\re a_1+|a|^2)(2\re b_1+|b|^2)}{|\langle a, b\rangle+a_1+\overline b_1|^2}=$$$$=\frac{\left(2\re a_1+\frac{(\re a_1)^2}{k^2_a}\right)\left(2\re a_1+\frac{(\re a_1)^2}{k^2_b}\right)}{|\langle a, b\rangle+2\re a_1+i\im a_1-i\im b_1|^2}\geq\frac{(\re a_1)^2\left(2+\frac{\re a_1}{k^2_a}\right)\left(2+\frac{\re a_1}{k^2_b}\right)}{2|\langle a, b\rangle+i\im a_1-i\im b_1|^2+2|2\re a_1|^2}$$$$\geq\frac{(\re a_1)^2\left(2+\frac{\re a_1}{k^2_a}\right)\left(2+\frac{\re a_1}{k^2_b}\right)}{2(|a||b|+|a|+|b|)^2+8(\re a_1)^2}=\frac{(\re a_1)^2\left(2+\frac{\re a_1}{k^2_a}\right)\left(2+\frac{\re a_1}{k^2_b}\right)}{2\left(\frac{(-\re a_1)^2}{k^2_ak^2_b}-\frac{\re a_1}{k_a}-\frac{\re a_1}{k_b}\right)^2+8(\re a_1)^2}$$$$=\frac{\left(2+\frac{\re a_1}{k^2_a}\right)\left(2+\frac{\re a_1}{k^2_b}\right)}{2\left(\frac{-\re a_1}{k^2_ak^2_b}+\frac{1}{k_a}+\frac{1}{k_b}\right)^2+8}>\frac{1}{2(1+2/k)^2+8}.$$ This finishes the proof.
\end{proof}

\medskip

Assume that we are in the settings of Proposition~\ref{1} and $D$ has real analytic boundary. Our aim is to replace $W$ with a neighborhood of $\ov D$.

\begin{remm}\label{przed34}
For $\zeta_0\in\DD_f$ we have $G(f(\zeta_0),\zeta_0)=0$ and $\frac{\partial G}{\partial\zeta}(f(\zeta_0),\zeta_0)=-1$. By the Implicit Function Theorem there exist neighborhoods $U_{\zeta_0},V_{\zeta_0}$ of $f(\zeta_0),\zeta_0$ respectively and a holomorphic function $F_{\zeta_0}:U_{\zeta_0}\longrightarrow V_{\zeta_0}$ such that for any $z\in U_{\zeta_0}$ the point $F_{\zeta_0}(\zeta)$ is the unique solution of the equation $G(z,\zeta)=0$, $\zeta\in V_{\zeta_0}$.

In particular, if $\zeta_0\in\DD$ then $F_{\zeta_0}=F$ near $f(\zeta_0)$.
\end{remm}

\begin{propp}\label{34}
Let $f:\DD\longrightarrow D$ be an $E$-mapping. Then there exist arbitrarily small neighborhoods $U$, $V$ of $\overline D$, $\CDD$ respectively such that for any $z\in U$ the equation $G(z,\zeta)=0$, $\zeta\in V$, has exactly one solution.
\end{propp}
\begin{proof} In view of Proposition~\ref{1} and Remark~\ref{przed34} it suffices to prove that there exist neighborhoods $U$, $V$ of $\overline D$, $\CDD$ respectively such that for any $z\in U$ the equation $G(z,\cdotp)=0$ has at most one solution $\zeta\in V$.

Assume the contrary. Then for any neighborhoods $U$ of $\overline D$ and $V$ of $\CDD$ there are $z\in U$, $\zeta_1,\zeta_2\in V$, $\zeta_1\neq\zeta_2$ such that $G(z,\zeta_1)=G(z,\zeta_2)=0$. For $m\in\NN$ put $$U_m:=\{z\in\CC^n:\dist(z,D)<1/m\},$$ $$V_m:=\{\zeta\in\CC:\dist(\zeta,\DD)<1/m\}.$$ There exist $z_m\in U_m$, $\zeta_{m,1},\zeta_{m,2}\in V_m$, $\zeta_{m,1}\neq\zeta_{m,2}$ such that $G(z_m,\zeta_{m,1})=G(z_m,\zeta_{m,2})=0$. Passing to a subsequence we may assume that $z_m\to z_0\in\ov D$. Analogously we may assume $\zeta_{m,1}\to\zeta_1\in \CDD$ and $\zeta_{m,2}\to\zeta_2\in\CDD$. Clearly, $G(z_0,\zeta_1)=G(z_0,\zeta_2)=0$. Let us consider few cases.

1) If $\zeta_1,\zeta_2\in\TT$, then $G(z_0,\zeta_j)=0$ is equivalent to $$\langle z_0-f(\zeta_j), \nu_D(f(\zeta_j))\rangle=0,\ j=1,2,$$ consequently $z_0-f(\zeta_j)\in T^{\CC}_D(f(\zeta_j))$. By the strong linear convexity of $D$ we get $z_0=f(\zeta_j)$. But $f$ is injective in $\CDD$, so $\zeta_1=\zeta_2=:\zeta_0$. It follows from Remark~\ref{przed34} that in a sufficiently small neighborhood of $(z_0,\zeta_0)$ all solutions of the equation $G(z,\zeta)=0$ are of the form $(z,F_{\zeta_0}(z))$. Points $(z_m,\zeta_{m,1})$ and $(z_m,\zeta_{m,2})$ belong to this neighborhood for large $m$, which gives a contradiction.

2) If $\zeta_1\in\TT$ and $\zeta_2\in\DD$, then analogously as above we deduce that $z_0=f(\zeta_1)$. Let us take an arbitrary sequence $\{\eta_m\}\su\DD$ convergent to $\zeta_1$. Then $f(\eta_m) \in D$ and $f(\eta_m)\to z_0$, so the sequence $G(f(\eta_m),\cdotp)$ converges to $G(z_0,\cdotp)$ uniformly on $\DD$. Since $G(z_0,\cdotp)\not\equiv 0$, $G(z_0,\zeta_2)=0$ and $\zeta_2\in\DD$, we deduce from Hurwitz Theorem that for large $m$ the functions $G(f(\eta_m),\cdotp)$ have roots $\theta_m\in\DD$ such that $\theta_m\to\zeta_2$. Hence $G(f(\eta_m),\theta_m)=0$ and from the uniqueness of solutions in $D\times\DD$  (Proposition~\ref{1}) we have $$\theta_m=F(f(\eta_m))=\eta_m.$$ This is a contradiction, because the left side tends to $\zeta_2$ and the right one to $\zeta_1$, as $m\to\infty$.

3) We are left with the case $\zeta_1,\zeta_2\in\DD$.
If $z_0\in\overline{D}\setminus f(\TT)$ then $z_0\in W$. In $W\times\DD$ all solutions of the equation $G=0$ are of the form $(z,F(z))$, $z\in W$. But for large $m$ the points $(z_m,\zeta_{m,1})$, $(z_m,\zeta_{m,2})$ belong to $W\times\DD$, which is a contradiction with the uniqueness.

If $z_0\in f(\TT)$, then $z_0=f(\zeta_0)$ for some $\zeta_0\in\TT$. Clearly, $G(f(\zeta_0),\zeta_0)=0$, whence $G(z_0,\zeta_0)=G(z_0,\zeta_1)=0$ and $\zeta_0\in\TT$, $\zeta_1\in \DD$. This is just the case 2), which has been already considered.
\end{proof}

\begin{corr} There are neighborhoods $U$, $V$ of $\overline D$ and $\CDD$ respectively with $V\Subset\DD_f$, such that the function $F$ extends holomorphically on $U$. Moreover, all solutions of the equation $G|_{U\times V}=0$ are of the form $(z,F(z))$, $z\in U$.

In particular, $F\circ f=\id_{V}$.
\end{corr}

\section{H\"older estimates}\label{22}

\begin{df}\label{30} For a given $c>0$ let the family $\mathcal{D}(c)$ consist of all pairs $(D,z)$, where $D\su\CC^n$, $n\geq 2$, is a bounded pseudoconvex domain with real $\mathcal C^2$ boundary and $z\in D$, satisfying
\begin{enumerate}
\item $\dist(z,\partial D)\geq 1/c$;
\item the diameter of $D$ is not greater than $c$ and $D$ satisfies the interior ball condition with a radius $1/c$;
\item for any $x,y\in D$ there exist $m\leq 8 c^2$ and open balls $B_0,\ldots,B_m\subset D$ of radius $1/(2c)$ such that $x\in B_0$, $y\in B_m$ and the distance between the centers of the balls $B_j$, $B_{j+1}$ is not greater than $1/(4c)$ for $j=0,\ldots,m-1$;
\item for any open ball $B\subset\mathbb{C}^n$ of radius not greater than $1/c$, intersecting non-emptily with $\pa D$, there exists a mapping $\Phi\in\OO(\overline{D},\mathbb{C}^n)$ such that
\begin{enumerate}
\item for any $w\in\Phi(B\cap\partial D)$ there is a ball of radius $c$ containing $\Phi(D)$ and tangent to $\partial\Phi(D)$ at $w$ (let us call it the ``exterior ball condition'' with a radius $c$);
\item $\Phi$ is biholomorphic in a neighborhood of $\ov B$ and $\Phi^{-1}(\Phi(B))=B$;
\item entries of all matrices $\Phi'$ on $B\cap\ov D$ and $(\Phi^{-1})'$ on $\Phi(B\cap\overline{D})$ are bounded in modulus by $c$;
\item $\dist(\Phi(z),\partial\Phi(D))\geq 1/c$;
\end{enumerate}
\item the normal vector $\nu_D$ is Lipschitz with a constant $2c$, that is $$|\nu_D(a)-\nu_D(b)|\leq 2c|a-b|,\ a,b\in \partial D;$$
\item the $\eps$-hull of $D$, i.e. a domain $D_{\eps}:=\{w\in\mathbb C^n:\dist (w,D)<\eps\}$, is strongly pseudoconvex for any $\eps\in (0,1/c).$
\end{enumerate}
\end{df}

Recall that the {\it interior ball condition} with a radius $r>0$ means that for any point $a\in\pa D$ there is $a'\in D$ and a ball $B_n(a',r)\su D$ tangent to $\pa D$ at $a$. Equivalently $$D=\bigcup_{a'\in D'}B_n(a',r)$$ for some set $D'\su D$.

It may be shown that (2) and (5) may be expressed in terms of boundedness of the normal curvature, boundedness of a domain and the condition (3). This however lies beyond the scope of this paper and needs some very technical arguments so we omit the proof of this fact. The reasons why we decided to use (2) in such a form is its connection with the condition (3) (this allows us to simplify the proof in some places).

\begin{rem}\label{con}
Note that any convex domain satisfying conditions (1)-...-(4) of Definition~\ref{30} satisfies conditions (5) and (6), as well.

Actually, it follows from (2) that for any $a\in\pa D$ there exists a ball $B_n(a',1/c)\su D$ tangent to $\pa D$ at $a$. Then $$\nu_D(a)=\frac{a'-a}{|a'-a|}=c(a'-a).$$ Hence $$|\nu_D(a)-\nu_D(b)|=c|a'-a-b'+b|=c|a'-b'-(a-b)|\leq c|a'-b'|+c|a-b|.$$ Since $D$ is convex, we have $|a'-b'|\leq|a-b|$, which gives (5).

The condition (6) is also clear --- for any $\eps>0$ an $\eps$-hull of a strongly convex domain is strongly convex.
\end{rem}

\begin{rem}
For a convex domain $D$ the condition (3) of Definition \ref{30} amounts to the condition (2).

Indeed, for two points $x,y\in D$ take two balls of radius $1/(2c)$ containing them and contained in $D$. Then divide the interval between the centers of the balls into $[4c^2]+1$ equal parts and take balls of radius $1/(2c)$ with centers at the points of the partition.

Note also that if $D$ is strongly convex and satisfies the interior ball condition with a radius $1/c$ and the exterior ball condition with a radius $c$, one can take $\Phi:=\id_{\CC^n}$.
\end{rem}

\begin{rem}\label{D(c),4}
For a strongly pseudoconvex domain $D$ and $c'>0$ and for any $z\in D$ such that $\dist(z,\partial D)>1/c'$ there exists $c=c(c')>0$ satisfying $(D,z)\in\mathcal{D}(c)$.

Indeed, the conditions (1)-...-(3) and (5)-(6) are clear. Only (4) is non-trivial.

The construction of the mapping $\Phi$ amounts to the construction of Forn\ae ss peak functions. Actually, apply directly Proposition 1 from \cite{For} to any boundary point of $\partial D$ (obviously $D$ has a Stein neighborhood basis). This gives a covering of $\partial D$ with  a finite number of balls $B_j$, maps $\Phi_j\in\OO(\overline{D},\mathbb{C}^n)$ and strongly convex $C^\infty$-smooth domains $C_j$, $j=1,\ldots, N$, such that
\begin{itemize}\item $\Phi_j(D)\subset C_j$;
\item $\Phi_j(\ov D)\subset\ov C_j$;
\item $\Phi_j(B_j\setminus\ov D)\subset\mathbb C^n\setminus\ov C_j$;
\item $\Phi_j^{-1}(\Phi_j(B_j))=B_j$;
\item $\Phi_j|_{B_j}: B_j\longrightarrow \Phi_j(B_j)$ is biholomorphic.
\end{itemize} Therefore, one may choose $c>0$ such that every $C_j$ satisfies the exterior ball condition with $c$, i.e. for any $x\in \partial C_j$ there is a ball of radius $c$ containing $C_j$ and tangent to $\partial C_j$ at $x$, every ball of radius $1/c$ intersecting non-emptily with $\pa D$ is contained in some $B_j$ (here one may use a standard argument invoking the Lebesgue number) and the conditions (c), (d) are also satisfied (with $\Phi:=\Phi_j$).
\end{rem}

In this section we use the words `uniform', `uniformly' if $(D,z)\in \mathcal D(c)$. This means that estimates will depend only on $c$ and will be independent on $D$ and $z$ if $(D,z)\in\mathcal{D}(c)$ and on $E$-mappings of $D$ mapping $0$ to $z$. Moreover, in what follows we assume that $D$ is a strongly linearlu convex domain with real-analytic boundary.

\begin{prop}\label{7}
Let $f:(\mathbb{D},0)\longrightarrow(D,z)$ be an $E$-mapping. Then $$\dist(f(\zeta),\partial D)\leq C(1-|\zeta|),\ \zeta\in\CDD$$ with $C>0$ uniform if $(D,z)\in\mathcal{D}(c)$.
\end{prop}
\begin{proof} There exists a uniform $C_1$ such that $$\text{if }\dist(w,\partial D)\geq 1/c\text{ then }k_D(w,z)<C_1.$$ Indeed, let $\dist(w,\partial D)\geq 1/c$ and let balls $B_0,\ldots,B_m$ with centers $b_0,\ldots,b_m$ be chosen to the points $w$, $z$ as in the condition (3) of Definition~\ref{30}. Then
\begin{multline*}k_D(w,z)\leq
k_D(w,b_0)+\sum_{j=0}^{m-1}k_D(b_j,b_{j+1})+k_D(b_m,z)\leq\\\leq k_{B_n(w,1/c)}(w,b_0)+\sum_{j=0}^{m-1}k_{B_j}(b_j,b_{j+1})+k_{B_n(z,1/c)}(b_m,z)=\\=p\left(0,\frac{|w-b_0|}{1/c}\right)+\sum_{j=0}^{m-1}p\left(0,\frac{|b_j-b_{j+1}|}{1/(2c)}\right)+
p\left(0,\frac{|b_m-z|}{1/c}\right)\leq\\\leq(m+2)p\left(0,\frac{1}{2}\right)\leq(8c^2+2)p\left(0,\frac{1}{2}\right)=:C_1.
\end{multline*}

If $\zeta\in\mathbb{D}$ is such that
$\dist(f(\zeta),\partial D)\geq 1/c$ then $$k_D(f(0),f(\zeta))\leq
C_2-\frac{1}{2}\log\dist(f(\zeta),\partial D)$$ with a uniform $C_2:=C_1+\frac{1}{2}\log c$.

In the other case, i.e. when $\dist(f(\zeta),\partial D)<1/c$, denote by $\eta$ the nearest point to
$f(\zeta)$ lying on $\partial D$. Let $w\in D$ be a center of a ball $B$ of radius $1/c$
tangent to $\partial D$ at $\eta$. By the condition (2) of Definition~\ref{30} we have $B\subset D$. Hence
\begin{multline*}k_D(f(0),f(\zeta))\leq k_D(f(0),w)+k_D(w,f(\zeta))\leq\\\leq
C_1+k_B(w,f(\zeta))\leq C_1+\frac{1}{2}\log 2-\frac{1}{2}\log\left(1-\frac{|f(\zeta)-w|}{1/c}\right)=\\=C_1+\frac{1}{2}\log 2-\frac{1}{2}\log(c\dist(f(\zeta),\partial B))=C_3-\frac{1}{2}\log\dist(f(\zeta),\partial D)
\end{multline*}
with a uniform $C_3:=C_1+\frac{1}{2}\log\frac{2}{c}$.

We have obtained the same type estimates in both cases. On the other side, by
Corollary~\ref{5} $$k_D(f(0),f(\zeta))=p(0,\zeta)\geq-\frac{1}{2}\log(1-|\zeta|),$$ which finishes the proof.
\end{proof}

Recall that we have assumed that $\rho$ is of the form~\eqref{rho}.
\begin{prop}\label{9}
Let $f:(\mathbb{D},0)\longrightarrow(D,z)$ be an $E$-mapping. Then $$C_1<\rho(\zeta)^{-1}<C_2,\ \zeta\in\TT,$$ where
$C_1,C_2$ are uniform if $(D,z)\in\mathcal{D}(c)$.
\end{prop}
\begin{proof} For the upper estimate fix $\zeta_0\in\TT$. Set $B:=B_n(f(\zeta_0),1/c)$ and let $\Phi\in\OO(\overline{D},\mathbb{C}^n)$ be as in the condition (4) of Definition~\ref{30} for $B$. One can assume that $f(\zeta_0)=\Phi(f(\zeta_0))=0$ and $\nu_D(0)=\nu_{\Phi(D)}(0)=(1,0,\ldots,0)$. Then $\Phi(D)$ is contained in the half-space $\{w\in\mathbb{C}^n:\re w_1<0\}$. Putting $h:=\Phi\circ f$ we have $$h_1(\mathbb{D})\subset\{w_1\in\mathbb{C}:\re w_1<0\}.$$ In virtue of the Schwarz Lemma on the half-plane
\begin{equation}\label{schh1}|h_1'(t\zeta_0)|\leq\frac{-2\re h_1(t\zeta_0)}{1-|t\zeta_0|^2}.\end{equation}

Let $\delta$ be the signed boundary distance of $\Phi(D)$, i.e. $$\delta(x):=\begin{cases}-\dist(x,\partial\Phi(D)),\ x\in\Phi(D)\\\ \ \ \dist(x,\partial\Phi(D)),\ x\notin\Phi(D).\end{cases}$$ It is a defining function of $\Phi(D)$ in a neighborhood of $0$ (recall that $\Phi^{-1}(\Phi(B))=B$). Observe that $$\delta(x)=\delta(0)+\re\langle\nabla\delta(0), x\rangle+O(|x|^2)=\re x_1+O(|x|^2).$$

If $x\in\Phi(D)$ tends transversally to $0$, then the angle between the vector $x$ and the hyperplane $\{w\in\mathbb{C}^n:\re w_1=0\}$ is separated from $0$, i.e. its sinus $(-\re x_1)/|x|>\varepsilon$ for some $\varepsilon>0$ independent on $x$. Thus $$\frac{\delta(x)}{\re x_1}=1+O(|x|)\text{ as }x\to 0\text{ transversally. }$$ Consequently \begin{equation}\label{50}-\re x_1\leq 2\dist(x,\partial\Phi(D))\text{ as }x\to 0\text{ transversally. }\end{equation}

We know that $t\longmapsto f(t\zeta_0)$ hits $\partial D$ transversally. Therefore, $t\longmapsto h(t\zeta_0)$ hits $\partial \Phi(D)$ transversally, as well. Indeed, we have \begin{multline}\label{hf}\left\langle\left.\frac{d}{dt}h(t\zeta_0)\right|_{t=1},\nu_{\Phi(D)}(h(\zeta_0))\right\rangle=\left\langle \Phi'(0)f'(\zeta_0)\zeta_0,\frac{(\Phi^{-1})'(0)^*\nabla r(0)}{|(\Phi^{-1})'(0)^*\nabla r(0)|}\right\rangle=\\=\frac{\langle\zeta_0 f'(\zeta_0),\nabla r(0)\rangle}{|(\Phi'(0)^{-1})^*\overline{\nabla r(0)}|}=\frac{\langle\zeta_0 f'(\zeta_0),\nu_D(f(\zeta_0))|\nabla r(0)|\rangle}{|(\Phi'(0)^{-1})^*\overline{\nabla r(0)}|}.
\end{multline}
where $r$ is a defining function of $D$. In particular,
\begin{multline*} \re \left\langle\left.\frac{d}{dt}h(t\zeta_0)\right|_{t=1},\nu_{\Phi(D)}(h(\zeta_0))\right\rangle=\re \frac{\langle\zeta_0 f'(\zeta_0),\nu_D(f(\zeta_0))|\nabla r(0)|\rangle}{|(\Phi'(0)^{-1})^*\overline{\nabla r(0)}|}=\\=\frac{\rho(\zeta_0)^{-1}|\nabla r(0)|}{|(\Phi'(0)^{-1})^*\overline{\nabla r(0)}|}\neq 0.\end{multline*}  This proves that $t\longmapsto h(t\zeta_0)$ hits $\partial\Phi(D)$ transversally.

Consequently, we may put $x=h(t\zeta_0)$ into \eqref{50} to get \begin{equation}\label{hf1}\frac{-2\re h_1(t\zeta_0)}{1-|t\zeta_0|^2}\leq\frac{4\dist(h(t\zeta_0),\partial\Phi(D))}
{1-|t\zeta_0|^2},\ t\to 1.\end{equation}
But $\Phi$ is a biholomorphism near $0$, so \begin{equation}\label{nfr}\frac{4\dist(h(t\zeta_0),\partial\Phi(D))}{1-|t\zeta_0|^2}\leq C_3\frac{\dist(f(t\zeta_0),\partial D)}{1-|t\zeta_0|},\ t\to 1,\end{equation} where $C_3$ is a uniform constant depending only on $c$ (thanks to the condition (4)(c) of Definition~\ref{30}). By Proposition \ref{7}, the term on the right side of~\eqref{nfr} does not exceed some uniform constant.

It follows from \eqref{hf} that \begin{multline*}\rho(\zeta_0)^{-1}=|\langle f'(\zeta_0)\zeta_0,\nu_D(f(\zeta_0))\rangle|\leq C_4|\langle h'(\zeta_0), \nu_{\Phi(D)}(h(\zeta_0))\rangle|=\\=C_4|h_1'(\zeta_0)|=\lim_{t\to 1}C_4|h_1'(t\zeta_0)|\end{multline*} with a uniform $C_4$ (here we use the condition (4)(c) of Definition~\ref{30} again).
Combining \eqref{schh1}, \eqref{hf1} and \eqref{nfr} we get the upper estimate for $\rho(\zeta_0)^{-1}.$

Now we are proving the lower estimate. Let $r$ be the signed boundary distance to $\partial D$. For $\varepsilon=1/c$ the function $$\varrho(w):=-\log(\varepsilon-r(w))+\log\varepsilon,\ w\in
D_\varepsilon,$$ where $D_\varepsilon$ is an $\varepsilon$-hull of $D$, is plurisubharmonic and defining for $D$. Indeed, we have $$-\log(\varepsilon-r(w))=-\log\dist(w,\partial D_\varepsilon),\ w\in D_\varepsilon$$ and $D_\varepsilon$ is pseudoconvex.

Therefore, a function $$v:=\varrho\circ f:\overline{\mathbb{D}}\longrightarrow(-\infty,0]$$ is subharmonic on $\DD$. Moreover, since $f$ maps $\TT$ in $\partial D$ we infer that $v=0$ on $\TT$. Moreover, since $|f(\lambda)-z|<c$ for $\lambda\in\mathbb{D}$, we have $$|f(\lambda)-z|<\frac{1}{2c}\text{ if }|\lambda|\leq\frac{1}{2c^2}.$$ Therefore, for a fixed $\zeta_0\in\TT$ $$M_{\zeta_0}(x):=\max_{t\in[0,2\pi]}v(\zeta_0 e^{x+it})\leq-\log\left(1+\frac{1}{2c\varepsilon}\right)=:-C_5\text{ if }x\leq-\log(2c^2).$$ Since $M_{\zeta_0}$ is convex for $x\leq 0$ and $M_{\zeta_0}(0)=0$, we get $$v(\zeta_0 e^x)\leq M_{\zeta_0}(x)\leq\frac{C_5x}{\log(2c^2)}\text{\ \ \
for \ }-\log(2c^2)\leq x\leq 0.$$ Hence (remember that $v(\zeta_0)=0$)
\begin{multline}\label{wk}\frac{C_5}{\log(2c^2)}\leq\left.\frac{d}{dx}v(\zeta_0 e^x)\right|_{x=0}=\sum_{j=1}^n\frac{\pa\varrho}{\pa z_j}(f(\zeta_0))f_j'(\zeta_0)\zeta_0=\\=\langle\zeta_0 f'(\zeta_0),\nabla\varrho(f(\zeta_0))\rangle=\rho(\zeta_0)^{-1}|\nabla\varrho(f(\zeta_0))|.\end{multline}
Moreover,
\begin{multline*}|\nabla\varrho(f(\zeta_0))|= \left\langle\nabla\varrho(f(\zeta_0)),\frac{\nabla\varrho(f(\zeta_0))}{|\nabla\varrho(f(\zeta_0))|}\right\rangle_\RR
=\langle\nabla\varrho(f(\zeta_0)),\nu_D(f(\zeta_0))\rangle_\RR=\\=\frac{\partial\varrho}{\partial\nu_D}(f(\zeta_0))=\lim_{t\to 0}\frac{\varrho(f(\zeta_0)+t\nu_D(f(\zeta_0)))-\varrho(f(\zeta_0))}{t}=\frac{1}{\varepsilon}=c,
\end{multline*} as $r(a+t\nu(a))=t$ if $a\in \partial D$ and $t\in\mathbb R$ is small enough.
This, together with \eqref{wk}, finishes the proof of the lower estimate.
\end{proof}

\begin{prop}\label{8}
Let $f:(\DD,0)\longrightarrow (D,z)$ be an $E$-mapping.
Then $$|f(\zeta_1)-f(\zeta_2)|\leq C\sqrt{|\zeta_1-\zeta_2|},\ \zeta_1,\zeta_2\in\CDD,$$ where $C$ is uniform if $(D,z)\in\mathcal{D}(c)$.
\end{prop}
\begin{proof}
Let $\zeta_0\in\DD$ be such that $1-|\zeta_0|<1/(cC)$, where $C$ is as in Proposition~\ref{7}. Then $B:=B_n(f(\zeta_0),1/c)$ intersects $\partial D$. Take $\Phi$ for the ball $B$ from the condition (4) of Definition~\ref{30}. Let $w$ denote the nearest point to $\Phi(f(\zeta_0))$ lying on $\partial\Phi(D)$.  From the conditions (4)(b)-(c) of Definition~\ref{30} we find that there is a uniform constant $r<1$ such that the point $w$ belongs to $\Phi(B\cap\partial D)$ provided that $|\zeta_0|\geq r$.

From the condition (4)(a) of Definition~\ref{30} we get that there is $w_0$ such that $\Phi(D)\subset B_n(w_0,c)$ and the ball $B_n(w_0,c)$ is tangent to $\Phi(D)$ at $w$. Let $$h(\zeta):=(\Phi\circ f)\left(\frac{\zeta_0-\zeta}{1-\overline{\zeta_0}\zeta}\right),\ \zeta\in\DD.$$
Then $h$ is holomorphic, $h(\DD)\subset B_n(w_0,c)$ and $h(0)=\Phi(f(\zeta_0))$. Using Lemma \ref{schw} we get \begin{multline*}|h'(0)|\leq\sqrt{c^2-|h(0)-w_0|^2}\leq\sqrt{2c(c-|\Phi(f(\zeta_0))-w_0|)}=\\
=\sqrt{2c(|w_0-w|-|\Phi(f(\zeta_0))-w_0|)}\leq\sqrt{2c}\sqrt{|\Phi(f(\zeta_0))-w|}=\\
=\sqrt{2c}\sqrt{\dist(\Phi(f(\zeta_0)),\partial\Phi(D))}.\end{multline*} Since $$h'(0)=\Phi'(f(\zeta_0))f'(\zeta_0)\left.\frac{d}{d\zeta}\frac{\zeta_0-\zeta}{1-\overline{\zeta_0}\zeta}\right|_{\zeta=0},$$ bby the condition (4)(c) of Definition~\ref{3} we get $$|h'(0)|\geq C_1|f'(\zeta_0)|(1-|\zeta_0|^2)$$ with a uniform $C_1$, so $$|f'(\zeta_0)|\leq\frac{|h'(0)|}{C_1(1-|\zeta_0|^2)}\leq\frac{\sqrt{2c}}{C_1}\frac{\sqrt{\dist(\Phi(f(\zeta_0)),\partial\Phi(D))}}{1-|\zeta_0|^2}\leq C_2\frac{\sqrt{\dist(f(\zeta_0),\partial D)}}{1-|\zeta_0|^2},$$ where $C_2$ is uniform. Combining with Proposition \ref{7} \begin{equation}\label{46}|f'(\zeta_0)|\leq C_3\frac{\sqrt{1-|\zeta_0|}}{1-|\zeta_0|^2}=\frac{C_3}{\sqrt{1-|\zeta_0|}},\end{equation} where a constant $C_3$ is uniform.

We have shown that \eqref{46} holds for $r\leq |\zeta_0|<1$ with a uniform $r<1$. For $|\zeta_0|<r$ we estimate in the following way $$|f'(\zeta_0)|\leq\max_{|\zeta|=r}|f'(\zeta)|\leq\frac{C_3}{\sqrt{1-r}}\leq\frac{C_4}{\sqrt{1-|\zeta_0|}}$$ with a uniform $C_4:=C_3/\sqrt{1-r}$.

Using Theorems \ref{lit1} and \ref{lit2} with $\alpha=1/2$ we finish the proof.
\end{proof}

\begin{prop}\label{10a}
Let $f:(\DD ,0)\longrightarrow (D,z)$ be an $E$-mapping.
Then $$|\rho(\zeta_1)-\rho(\zeta_2)|\leq C\sqrt{|\zeta_1-\zeta_2|},\ \zeta_1,\zeta_2\in\TT,$$ where
$C$ is uniform if $(D,z)\in\mathcal{D}(c)$.
\end{prop}
\begin{proof}It suffices to prove that there exist uniform $C,C_1>0$ such that $$|\rho(\zeta_1)-\rho(\zeta_2)|\leq C\sqrt{|\zeta_1-\zeta_2|},\ \zeta_1,\zeta_2\in\TT,\ |\zeta_1-\zeta_2|<C_1.$$

Fix $\zeta_1\in\TT$. Without loss of generality we may assume that $\nu_{D,1}(f(\zeta_1))=1$. Let $0<C_1\leq 1/4$ be uniform and such that $$|\nu_{D,1}(f(\zeta))-1|<1/2,\ \zeta\in\TT\cap B_n(\zeta_1,3C_1).$$ It is possible, since by Proposition \ref{8} $$|{\nu_D(f(\zeta))}-{\nu_D(f(\zeta'))}|\leq 2c|f(\zeta)-f(\zeta')|\leq C'\sqrt{|\zeta-\zeta'|},\ \zeta,\zeta'\in\TT,$$ with a uniform $C'>0$. There exists a function $\psi\in\cC^1(\TT,[0,1])$ such that $\psi=1$ on $\TT\cap B_n(\zeta_1,2C_1)$ and $\psi=0$ on $\TT\setminus B_n(\zeta_1,3C_1)$. Then the function $\phi:\TT\longrightarrow\CC$ defined by $$\varphi:=(\overline{\nu_{D,1}\circ f}-1)\psi+1$$ satisfies
\begin{enumerate}
\item $\varphi(\zeta)=\overline{\nu_{D,1}(f(\zeta))}$, $\zeta\in\TT\cap B_n(\zeta_1,2C_1)$;
\item $|\varphi(\zeta)-1|<1/2$, $\zeta\in\TT$;
\item $\phi$ is uniformly $1/2$-H\"older continuous on $\TT$, i.e. it is $1/2$-H\"older continuous with a uniform constant (remember that $\psi$ was chosen uniformly).
\end{enumerate}

First observe that $\log\varphi$ is well-defined. Using using properties listed above we deduce that $\log\varphi$ and $\im\log\varphi$ are uniformly $1/2$-H\"older continuous on $\TT$, as well. The function $\im\log\varphi$ can be extended continuously to a function $v:\CDD\longrightarrow\RR$, harmonic in $\DD$. There is a function $h\in\mathcal O(\DD)$ such that $v=\im h$ in $\DD$. Taking $h-\re h(0)$ instead of $h$, one can assume that $\re h(0)=0$. By Theorem \ref{priv} applied to $ih$, we get that the function $h$ extends continuously on $\CDD$ and $h$ is uniformly $1/2$-H\"older continuous in $\CDD$. Hence the function $u:=\re h:\CDD\longrightarrow\RR$ is uniformly $1/2$-H\"older continuous in $\CDD$ with a uniform constant $C_2$. Furthermore, $u$ is uniformly bounded in $\CDD$, since $$|u(\zeta)|=|u(\zeta)-u(0)|\leq C_2\sqrt{|\zeta|},\ \zeta\in\CDD.$$

Let $g(\zeta):=\wi{f}_1(\zeta)e^{-h(\zeta)}$ and $G(\zeta):=g(\zeta)/\zeta$. Then $g\in\mathcal O(\DD)\cap\mathcal C(\overline{\DD})$ and $G\in\mathcal O(\DD_*)\cap\mathcal C((\overline{\DD})_*)$. Note that for $\zeta\in\TT$ $$|g(\zeta)|=|\zeta
\rho(\zeta)\overline{\nu_{D,1}(f(\zeta))}e^{-h(\zeta)}|\leq\rho(\zeta)e^{-u(\zeta)},$$ which, combined with
Proposition \ref{9}, the uniform boundedness of $u$ and the maximum principle, gives a uniform boundedness of $g$ in $\CDD$. The function $G$ is uniformly bounded in $\overline{\DD}\cap B_n(\zeta_1,2C_1)$. Moreover, for $\zeta\in\TT\cap B_n(\zeta_1,2C_1)$ \begin{eqnarray*} G(\zeta)&=&\rho(\zeta)\overline{\nu_{D,1}(f(\zeta))}e^{-u(\zeta)-i\im\log \phi(\zeta)}=\\&=&\rho(\zeta)\overline{\nu_{D,1}(f(\zeta))}e^{-u(\zeta)+\re\log\phi(\zeta)}e^{-\log\phi(\zeta)}
=\rho(\zeta)e^{-u(\zeta)+\re\log\phi(\zeta)}\in\mathbb{R}.\end{eqnarray*} By the Reflection Principle one can extend $G$ holomorphically past $\TT\cap B_n(\zeta_1,2C_1)$ to a function (denoted by the same letter) uniformly bounded in $B_n(\zeta_1,2C_2)$, where a constant $C_2$ is uniform. Hence, from the Cauchy formula, $G$ is uniformly Lipschitz continuous in $B_n(\zeta_1,C_2)$, consequently uniformly $1/2$-H\"older continuous in $B_n(\zeta_1,C_2)$.

Finally, the functions $G$, $h$, $\nu_{D,1}\circ f$ are uniformly $1/2$-H\"older continuous on $\TT\cap B_n(\zeta_1,C_2)$, $|\nu_{D,1}\circ f|>1/2$ on $\TT\cap B_n(\zeta_1,C_2)$, so the function $\rho=Ge^h/\overline{\nu_{D,1}\circ f}$ is uniformly $1/2$-H\"older continuous on $\TT\cap B_n(\zeta_1,C_2)$.
\end{proof}

\begin{prop}\label{10b}
Let $f:(\DD,0)\longrightarrow (D,z)$ be an $E$-mapping.
Then $$|\wi{f}(\zeta_1)-\wi{f}(\zeta_2)|\leq C\sqrt{|\zeta_1-\zeta_2|},\ \zeta_1,\zeta_2\in\overline{\DD},$$ where
$C$ is uniform if $(D,z)\in\mathcal{D}(c)$.

\end{prop}
\begin{proof}
By Propositions \ref{8} and \ref{10a} we have desired inequality for $\zeta_1,\zeta_2\in\TT$.  Theorem \ref{lit2} finishes the proof.
\end{proof}

\section{Openness of $E$-mappings' set}\label{27}
We shall show that perturbing a little a domain $D$ equipped with an $E$-mapping, we obtain a domain which also has an $E$-mapping, being close to a given one.

\subsection{Preliminary results}

\begin{propp}\label{11}
Let $f:\mathbb{D}\longrightarrow D$ be an $E$-mapping. Then there exist domains $G,\wi D,\wi G\subset\CC^n$ and a biholomorphism $\Phi:\wi D\longrightarrow\wi G$ such that
\begin{enumerate}
\item $\wi D,\wi G$ are neighborhoods of $\overline D,\overline G$ respectively;
\item $\Phi(D)=G$;
\item $g(\zeta):=\Phi(f(\zeta))=(\zeta,0,\ldots,0),\ \zeta\in\CDD$;
\item $\nu_G(g(\zeta))=(\zeta,0,\ldots,0),\ \zeta\in\TT$;
\item for any $\zeta\in\TT$, a point $g(\zeta)$ is a point of the strong linear convexity of $G$.
\end{enumerate}
\end{propp}
\begin{proof}
Let $U,V$ be the sets from Proposition \ref{34}. We claim that after a linear change of coordinates one can assume that $\widetilde{f}_1,\widetilde{f}_2$ do not have common zeroes in $V$.

Since $ f'\bullet\widetilde{f}=1$, at least one of the functions $\wi f_1,\ldots,\wi f_n$, say $\wi f_1$, is not identically equal to $0$. Let $\lambda_1,\ldots,\lambda_m$ be all zeroes of $\wi f_1$ in $V$. We may find $\alpha\in\CC^n$ such that $$(\alpha_1\wi f_1+\ldots+\alpha_n\wi f_n)(\lambda_j)\neq 0,\ j=1,\ldots,m.$$ Otherwise, for any $\alpha\in\CC^n$ there would exist $j\in\{1,\ldots,m\}$ such that $\alpha\bullet\wi f(\lambda_j)=0$, hence $$\CC^n=\bigcup_{j=1}^m\{\alpha\in\CC^n:\ \alpha\bullet\wi f(\lambda_j)=0\}.$$ The sets $\{\alpha\in\CC^n:\alpha \bullet \wi f(\lambda_j)=0\}$, $j=1,\ldots,m$, are the $(n-1)$-dimensional complex hyperplanes, so their finite sum cannot be the space $\CC^n$.

Of course, at least one of the numbers $\alpha_2,\ldots,\alpha_n$, say $\alpha_2$, is non-zero. Let
$$A:=\left[\begin{matrix}
1 & 0 & 0 & \cdots & 0\\
\alpha_1 & \alpha_2 & \alpha_3 &\cdots & \alpha_n\\
0 & 0 & 1 & \cdots & 0\\
\vdots & \vdots & \vdots &\ddots & \vdots \\
0 & 0 & 0 & \cdots & 1
\end{matrix}\right],\quad B:=(A^T)^{-1}.$$ We claim that $B$ is a change of coordinates we are looking for. If $r$ is a defining function of $D$ then $r\circ B^{-1}$ is a defining function of $B_n(D)$, so $B_n(D)$ is a bounded strongly linearly convex domain with real analytic boundary. Let us check that $Bf$ is an $E$-mapping of $B_n(D)$ with associated mappings \begin{equation}\label{56}A\wi f\in\OO(\CDD)\text{\ \  and\ \ }\rho\frac{|A\overline{\nabla r\circ f}|}{|\nabla r\circ f|}\in\CLW(\TT).\end{equation} The conditions (1) and (2) of Definition~\ref{21} are clear. For $\zeta\in\TT$ we have \begin{equation}\label{57}\overline{\nu_{B_n(D)}(Bf(\zeta))}=\frac{\overline{\nabla(r\circ B^{-1})(Bf(\zeta))}}{|\nabla(r\circ B^{-1})(Bf(\zeta))|}=\frac{(B^{-1})^T\overline{\nabla r(f(\zeta))}}{|(B^{-1})^T\overline{\nabla r(f(\zeta))}|}=\frac{A\overline{\nabla r(f(\zeta))}}{|A\overline{\nabla r(f(\zeta))}|},\end{equation} so
\begin{equation}\label{58}\zeta\rho(\zeta)\frac{|A\overline{\nabla r(f(\zeta))}|}{|\nabla r(f(\zeta))|}\overline{\nu_{B_n(D)}(Bf(\zeta))}=\zeta\rho(\zeta)A\overline{\nu_D(f(\zeta))}=A\wi f(\zeta).\end{equation} Moreover, for $\zeta\in\TT$, $z\in D$ \begin{multline*}\langle Bz-Bf(\zeta), \nu_{B_n(D)}(Bf(\zeta))\rangle=\overline{\nu_{B_n(D)}(Bf(\zeta))}^T(Bz-Bf(\zeta))=\\=\frac{\overline{\nabla r(f(\zeta))}^TB^{-1}B_n(z-f(\zeta))}{|(B^{-1})^T\overline{\nabla r(f(\zeta))}|}=\frac{|\nabla r(f(\zeta))|}{|(B^{-1})^T\overline{\nabla r(f(\zeta))}|}\overline{\nu_D(f(\zeta))}^T(z-f(\zeta))=\\=\frac{|\nabla r(f(\zeta))|}{|(B^{-1})^T\overline{\nabla r(f(\zeta))}|}\langle z-f(\zeta), \nu_D(f(\zeta))\rangle.
\end{multline*}
Therefore, $B$ is a desired linear change of coordinates, as claimed.

If necessary, we shrink the sets $U,V$ associated with $f$ to sets associated with $Bf$. There exist holomorphic mappings $h_1,h_2:V\longrightarrow\mathbb{C}$ such that
$$h_1\widetilde{f}_1+h_2\widetilde{f}_2\equiv 1\text{ in }V.$$ Generally, it is a well-known fact for functions on pseudoconvex domains, however in this case it may be shown quite elementarily. Indeed, if $\widetilde{f}_1\equiv 0$ or $\widetilde{f}_2\equiv 0$ then it is obvious. In the opposite case, let $\widetilde{f}_j=F_jP_j$, $j=1,2$, where $F_j$ are holomorphic, non-zero in $V$ and $P_j$ are polynomials with all (finitely many) zeroes in $V$. Then $P_j$ are relatively prime, so there are polynomials $Q_j$, $j=1,2$, such that $$Q_1P_1+Q_2P_2\equiv 1.$$ Hence $$\frac{Q_1}{F_1}\widetilde{f}_1+\frac{Q_2}{F_2}\widetilde{f}_2\equiv 1\ \text{ in }V.$$

Consider the mapping $\Psi:V\times\mathbb{C}^{n-1}\longrightarrow\mathbb{C}^n$ given by
\begin{equation}\label{et2}
\Psi_1(Z):=f_1(Z_1)-Z_2\widetilde{f}_2(Z_1)-h_1(Z_1)
\sum_{j=3}^{n}Z_j\widetilde{f}_j(Z_1),
\end{equation}
\begin{equation}\label{et3}
\Psi_2(Z):=f_2(Z_1)+Z_2\widetilde{f}_1(Z_1)-h_2(Z_1)
\sum_{j=3}^{n}Z_j\widetilde{f}_j(Z_1),
\end{equation}
\begin{equation}\label{et4}
\Psi_j(Z):=f_j(Z_1)+Z_j,\ j=3,\ldots,n.
\end{equation}

We claim that $\Psi$ is biholomorphic in $\Psi^{-1}(U)$. First of all observe that $\Psi^{-1}(\{z\})\neq\emptyset$ for any $z\in U$. Indeed, by Proposition \ref{34} there exists (exactly one) $Z_1\in V$ such that $$(z-f(Z_1))\bullet\widetilde{f}(Z_1)=0.$$ The numbers $Z_j\in\CC$, $j=3,\ldots,n$ are determined uniquely by the equations $$Z_j=z_j-f_j(Z_1).$$ At least one of the numbers $\wi f_1(Z_1),\wi f_2(Z_1)$, say $\wi f_1(Z_1)$, is non-zero. Let $$Z_2:=\frac{z_2-f_2(Z_1)+h_2(Z_1)\sum_{j=3}^{n}Z_j\widetilde{f}_j(Z_1)}{\wi f_1(Z_1)}.$$ Then we easily check that the equality $$z_1=f_1(Z_1)-Z_2\widetilde{f}_2(Z_1)-h_1(Z_1)
\sum_{j=3}^{n}Z_j\widetilde{f}_j(Z_1)$$ is equivalent to $(z-f(Z_1))\bullet\widetilde{f}(Z_1)=0$, which is true.

To finish the proof of biholomorphicity of $\Psi$ in $\Psi^{-1}(U)$ it suffices to check that $\Psi$ is injective in $\Psi^{-1}(U)$. Let us take $Z,W$ such that $\Psi(Z)=\Psi(W)=z\in U$. By a direct computation both $\zeta=Z_1\in V$ and $\zeta=W_1\in V$ solve the equation
$$(z-f(\zeta))\bullet\widetilde{f}(\zeta)=0.$$ From Proposition \ref{34} we infer that it has exactly one solution. Hence $Z_1=W_1$. By \eqref{et4} we have $Z_j=W_j$ for $j=3,\ldots,n$. Finally $Z_2=W_2$ follows from
one of the equations \eqref{et2}, \eqref{et3}. Let $G:=\Psi^{-1}(D)$, $\wi D:=U$, $\wi G:=\Psi^{-1}(U)$, $\Phi:=\Psi^{-1}$.

Now we are proving that $\Phi$ has desired properties. We have $$\Psi_j(\zeta,0,\ldots,0)=f_j(\zeta),\ j=1,\ldots,n,$$ so $\Phi(f(\zeta))=(\zeta,0,\ldots,0)$, $\zeta\in\CDD$. Put $g(\zeta):=\Phi(f(\zeta))$, $\zeta\in\CDD$. Note that the entries of the matrix $\Psi'(g(\zeta))$ are $$\frac{\partial\Psi_1}{\partial Z_1}(g(\zeta))=f_1'(\zeta),\ \frac{\partial\Psi_1}{\partial Z_2}(g(\zeta))=-\widetilde{f}_2(\zeta),\ \frac{\partial\Psi_1}{\partial Z_j}(g(\zeta))=-h_1(\zeta)\widetilde{f}_j(\zeta),\ j\geq 3,$$$$\frac{\partial\Psi_2}{\partial Z_1}(g(\zeta))=f_2'(\zeta),\ \frac{\partial\Psi_2}{\partial Z_2}(g(\zeta))=\widetilde{f}_1(\zeta),\ \frac{\partial\Psi_2}{\partial Z_j}(g(\zeta))=-h_2(\zeta)\widetilde{f}_j(\zeta),\ j\geq 3,$$$$\frac{\partial\Psi_k}{\partial Z_1}(g(\zeta))=f_k'(\zeta),\ \frac{\partial\Psi_k}{\partial Z_2}(g(\zeta))=0,\ \frac{\partial\Psi_k}{\partial Z_j}(g(\zeta))=\delta^{k}_{j},\ j,k\geq 3.$$ Thus $\Psi '(g(\zeta))^T\wi f(\zeta)=(1,0,\ldots,0)$, $\zeta\in\CDD$ (since $f'\bullet\wi f=1$). Let us take a defining function $r$ of $D$. Then $r\circ\Psi$ is a defining function of $G$. Therefore, \begin{multline*}\nu_G(g(\zeta))=\frac{\nabla(r\circ\Psi)(g(\zeta))}{|\nabla(r\circ\Psi)(g(\zeta))|}=
\frac{\overline{\Psi'(g(\zeta))}^T\nabla r(f(\zeta))}{|\overline{\Psi'(g(\zeta))}^T\nabla r(f(\zeta))|}=\\=\frac{\overline{\Psi'(g(\zeta))}^T\ov{\frac{\wi f(\zeta)}{\zeta\rho(\zeta)}}|\nabla r(f(\zeta))|}{\left|\overline{\Psi'(g(\zeta))}^T\ov{\frac{\wi f(\zeta)}{\zeta\rho(\zeta)}}|\nabla r(f(\zeta))|\right|}=g(\zeta),\ \zeta\in\TT.\end{multline*}

It remains to prove the fifth condition. By Definition \ref{29}(2) we have to show that \begin{equation}\label{sgf}\sum_{j,k=1}^n\frac{\partial^2(r\circ\Psi)}{\partial z_j\partial\overline{z}_k}(g(\zeta))X_{j}\overline{X}_{k}>\left|\sum_{j,k=1}^n\frac{\partial^2(r\circ\Psi)}{\partial z_j\partial z_k}(g(\zeta))X_{j}X_{k}\right|\end{equation} for $\zeta\in\TT$ and $X\in(\CC^{n})_*$ with
$$\sum_{j=1}^n\frac{\partial(r\circ\Psi)}{\partial z_j}(g(\zeta))X_{j}=0,$$ i.e. $X_1=0$. We have $$\sum_{j,k=1}^n\frac{\partial^2(r\circ\Psi)}{\partial z_j\partial\overline{z}_k}(g(\zeta))X_{j}\overline{X}_{k}=\sum_{j,k,s,t=1}^n\frac{\partial^2 r}{\partial z_s\partial\overline{z}_t}(f(\zeta))\frac{\partial\Psi_s}{\partial z_j}(g(\zeta))\overline{\frac{\partial\Psi_t}{\partial z_k}(g(\zeta))}X_{j}\overline{X}_{k}=$$$$=\sum_{s,t=1}^n\frac{\partial^2 r}{\partial z_s\partial\overline{z}_t}(f(\zeta))Y_{s}\overline{Y}_{t},$$ where $$Y:=\Psi'(g(\zeta))X.$$ Note that $Y\neq 0$. Additionally $$\sum_{s=1}^n\frac{\partial r}{\partial z_s}(f(\zeta))Y_{s}=\sum_{j,s=1}^n\frac{\partial r}{\partial z_s}(f(\zeta))\frac{\partial\Psi_s}{\partial z_j}(g(\zeta))X_j=\sum_{j=1}^n\frac{\partial(r\circ\Psi)}{\partial z_j}(g(\zeta))X_{j}=0.$$ Therefore, by the strong linear convexity of $D$ at $f(\zeta)$ $$\sum_{s,t=1}^n\frac{\partial^2 r}{\partial z_s\partial\overline{z}_t}(f(\zeta))Y_{s}\overline{Y}_{t}>\left|\sum_{s,t=1}^n\frac{\partial^2 r}{\partial z_s\partial z_t}(f(\zeta))Y_{s}Y_{t}\right|.$$ To finish the proof observe that $$\left|\sum_{j,k=1}^n\frac{\partial^2(r\circ\Psi)}{\partial z_j\partial z_k}(g(\zeta))X_{j}X_{k}\right|=\left|\sum_{j,k,s,t=1}^n\frac{\partial^2 r}{\partial z_s\partial z_t}(f(\zeta))\frac{\partial\Psi_s}{\partial z_j}(g(\zeta))\frac{\partial\Psi_t}{\partial z_k}(g(\zeta))X_{j}X_{k}+\right.$$$$\left.+\sum_{j,k,s=1}^n\frac{\partial r}{\partial z_s}(f(\zeta))\frac{\partial^2\Psi_s}{\partial z_j\partial z_k}(g(\zeta))X_{j}X_{k}\right|=$$$$=\left|\sum_{s,t=1}^n\frac{\partial^2 r}{\partial z_s\partial z_t}(f(\zeta))Y_{s}Y_{t}+\sum_{j,k=2}^n\sum_{s=1}^n\frac{\partial r}{\partial z_s}(f(\zeta))\frac{\partial^2\Psi_s}{\partial z_j\partial z_k}(g(\zeta))X_{j}X_{k}\right|$$ and $$\frac{\partial^2\Psi_s}{\partial z_j\partial z_k}(g(\zeta))=0,\ j,k\geq 2,\ s\geq 1,$$ which gives \eqref{sgf}.
\end{proof}

\begin{remm}\label{rem:theta}
Let $D$ be a bounded domain in $\mathbb C^n$ and let $f:\DD\longrightarrow D$ be a (weak) stationary mapping such that $\partial D$ is real analytic in a neighborhood of $f(\TT)$. Assume moreover that there are a neighborhood $U$ of $f(\CDD)$ and a mapping $\Theta:U\longrightarrow\CC^n$ biholomorphic onto its image and the set $D\cap U$ is connected. Then $\Theta\circ f$ is a (weak) stationary mapping of $G:=\Theta(D\cap U)$.

In particular, if $U_1$, $U_2$ are neighborhoods of the closures of domains $D_1$, $D_2$ with real analytic boundaries and  $\Theta:U_1\longrightarrow U_2$ is a biholomorphism such that $\Theta(D_1)=D_2$, then $\Theta$ maps (weak) stationary mappings of $D_1$ onto (weak) stationary mappings of $D_2$.
\end{remm}
\begin{proof}
Actually, it is clear that two first conditions of the definition of (weak) stationary mappings are preserved by $\Theta$. To show the third one we proceed similarly as in the equations \eqref{56}, \eqref{57}, \eqref{58}. Let $f:\DD\longrightarrow D $ be a (weak) stationary mapping. The candidates for the mappings in condition (3) (resp. (3')) of Definition~\ref{21} for $\Theta\circ f$ in the domain $G$ are $$((\Theta'\circ f)^{-1})^T\wi f\text{\ \  and\ \ }\rho\frac{|((\Theta'\circ f)^{-1})^T\overline{\nabla r\circ f}|}{|\nabla r\circ f|}.$$ Indeed, for $\zeta\in\TT$ \begin{multline*}\overline{\nu_{G}(\Theta(f(\zeta)))}=
\frac{\overline{\nabla(r\circ\Theta^{-1})(\Theta(f(\zeta)))}}{|\nabla(r\circ\Theta^{-1})(\Theta(f(\zeta)))|}=\frac{[(\Theta^{-1})'(\Theta(f(\zeta)))]^T\overline{\nabla r(f(\zeta))}}{|[(\Theta^{-1})'(\Theta(f(\zeta)))]^T\overline{\nabla r(f(\zeta))}|}=\\
=\frac{(\Theta'(f(\zeta))^{-1})^T\overline{\nabla r(f(\zeta))}}{|(\Theta'(f(\zeta))^{-1})^T\overline{\nabla r(f(\zeta))}|},
\end{multline*}
hence
\begin{multline*}\zeta\rho(\zeta)\frac{|(\Theta'(f(\zeta))^{-1})^T\overline{\nabla r(f(\zeta))}|}{|\nabla r(f(\zeta))|}\overline{\nu_{G}(\Theta(f(\zeta)))}=\\
=\zeta\rho(\zeta)(\Theta'(f(\zeta))^{-1})^T\overline{\nu_{D}(f(\zeta))}=
(\Theta'(f(\zeta))^{-1})^T\wi f(\zeta).
\end{multline*}
\end{proof}

\subsection{Situation (\dag)}\label{dag}
Consider the following situation, denoted by (\dag) (with data $D_0$ and $U_0$):
\begin{itemize}
\item $D_0$ is a bounded domain in $\CC^n$, $n\geq 2$;
\item $f_0:\CDD\ni\zeta\longmapsto(\zeta,0,\ldots,0)\in\ov D_0$, $\zeta\in\CDD$;
\item $f_0(\DD)\subset D_0$;
\item $f_0(\TT)\subset\partial D_0$;
\item $\nu_{D_0}(f_0(\zeta))=(\zeta,0,\ldots,0)$, $\zeta\in\TT$;
\item for any $\zeta\in\TT$, a point $f_0(\zeta)$ is a point of the strong linear convexity of $D_0$;
\item $\partial D_0$ is real analytic in a neighborhood $U_0$ of $f_0(\TT)$ with a function $r_0$;
\item $|\nabla r_0|=1$ on $f_0(\TT)$ (in particular, $r_{0z}(f_0(\zeta))=(\ov\zeta/2,0,\ldots,0)$, $\zeta\in\TT$).
\end{itemize}

Since $r_0$ is real analytic on $U_0\su\RR^{2n}$, it extends in a natural way to a holomorphic function in a neighborhood $U_0^\CC\su\mb{C}^{2n}$ of $U_0$. Without loss of generality we may assume that $r_0$ is bounded on $U_0^\CC$. Set $$X_0=X_0(U_0,U_0^{\mathbb C}):=\{r\in\mc{O}(U_0^\CC):\text{$r(U_0)\su\mb{R}$ and $r$ is bounded}\},$$ which equipped with the sup-norm is a (real) Banach space.

\begin{remm} Lempert considered the case when $U_0$ is a neighborhood of a boundary of a bounded domain $D_0$ with real analytic boundary. We shall need more general results to prove the `localization property'.
\end{remm}

\subsection{General lemmas}\label{General lemmas}
We keep the notation from Subsection \eqref{dag} and assume Situation (\dag).

Let us introduce some additional objects we shall be dealing with and let us prove more general lemmas (its generality will be useful in the next section).

Consider the Sobolev space $W^{2,2}(\TT)=W^{2,2}(\TT,\CC^m)$ of functions $f:\TT\longrightarrow\CC^m$, whose first two derivatives (in the sense of distribution) are in $L^2(\TT)$. The $W^{2,2}$-norm is denoted by $\|\cdot\|_W$. For the basic properties of $W^{2,2}(\TT)$ see Appendix.

Put $$B:=\{f\in W^{2,2}(\TT,\CC^n):f\text{ extends holomorphically on $\mb{D}$ and $f(0)=0$}\},$$$$B_0:=\{f\in B:f(\TT)\su U_0\},\quad B^*:=\{\overline{f}:f\in B\},$$$$Q:=\{q\in W^{2,2}(\TT,\CC):q(\TT)\su\RR\},\quad Q_0:=\{q\in Q:q(1)=0\}.$$

It is clear that $B$, $B^*$, $Q$ and $Q_0$ equipped with the norm $\|\cdot\|_W$ are (real) Banach spaces. Note that $B_0$ is an open neighborhood of $f_0$. In what follows, we identify $f\in B$ with its unique holomorphic extension on $\mb{D}$.

Let us define the projection $$\pi:W^{2,2}(\TT,\CC^n)\ni f=\sum_{k=-\infty}^{\infty}a_k\zeta^{k}\longmapsto\sum_{k=-\infty}^{-1}a_k\zeta^{k}\in{B^*}.$$ Note that $f\in W^{2,2}(\TT,\CC^n)$ extends holomorphically on $\mb{D}$ if and only if $\pi(f)=0$ (and the extension is $\mathcal C^{1/2}$ on $\TT$). Actually, it suffices to observe that
$g(\zeta):=\sum_{k=-\infty}^{-1}a_k\zeta^{k}$, $\zeta\in\TT$, extends holomorphically on $\DD$ if and only if $a_k=0$ for $k<0$. This follows immediately from the fact that the mapping $\TT\ni\zeta\longmapsto g(\ov\zeta)\in\CC^n$ extends holomorphically on $\DD$.

Consider the mapping $\Xi:X_0\times\mb{C}^n\times B_0\times
Q_0\times\mb{R}\longrightarrow Q\times{B^*}\times\mb{C}^n$ defined by
$$\Xi(r,v,f,q,\lambda):=(r\circ f,\pi(\zeta(1+q)(r_z\circ f)),f'(0)-\lambda v),$$ where $\zeta$ is treated as the identity function on $\TT$.

We have the following

\begin{lemm}\label{cruciallemma} There exist a neighborhood $V_0$ of $(r_0,f_0'(0))$ in $X_0\times\mb{C}^n$ and a real analytic mapping $\Upsilon:V_0\longrightarrow B_0\times Q_0\times\mb{R}$ such that for any $(r,v)\in V_0$ we have $\Xi(r,v,\Upsilon(r,v))=0$.
\end{lemm}
\bigskip
Let $\wi\Xi:X_0\times\mb{C}^n\times B_0\times Q_0\times(0,1)\longrightarrow Q\times{B^*}\times\mb{C}^n$ be defined as $$\wi\Xi(r,w,f,q,\xi):=(r\circ f,\pi(\zeta(1+q)(r_z\circ f)),f(\xi)-w).$$

Analogously we have
\begin{lemm}\label{cruciallemma1} Let $\xi_0\in(0,1)$. Then there exist a neighborhood $W_0$ of $(r_0,f_0(\xi_0))$ in $X_0\times D_0$ and a real analytic mapping $\wi\Upsilon:W_0\longrightarrow B_0\times Q_0\times(0,1)$ such that for any $(r,w)\in W_0$ we have $\wi\Xi(r,w,\wi\Upsilon(r,w))=0$.
\end{lemm}

\begin{proof}[Proof of Lemmas \ref{cruciallemma} and \ref{cruciallemma1}]

We will prove the first lemma. Then we will see that a proof of the second one reduces to that proof.

We claim that $\Xi$ is real analytic. The only problem is to show that the mapping $$T: X_0\times B_0\ni(r,f)\longmapsto r\circ f\in Q$$ is real analytic (the real analyticity of the mapping $X_0\times B_0\ni(r,f)\longmapsto r_z\circ f\in W^{2,2}(\TT,\CC^n)$ follows from this claim).

Fix $r\in X_0$, $f\in B_0$ and take $\eps>0$ so that a $2n$-dimensional polydisc $P_{2n}(f(\zeta),\eps)$ is contained in $U_0^\CC$ for any $\zeta\in\TT$. Then any function $\wi r\in X_0$ is holomorphic in $U_0^\CC$, so it may be expanded as a holomorphic series convergent in $P_{2n}(f(\zeta),\eps)$. Losing no generality we may assume that $n$-dimensional polydiscs $P_{n}(f(\zeta),\eps)$, $\zeta\in\TT$, satisfy $P_{n}(f(\zeta),\eps)\su U_0$. This gives an expansion of the function $\wi r$ at any point $f(\zeta)$, $\zeta\in\TT$, into a series $$\sum_{\alpha\in\NN_0^{2n}}\frac{1}{\alpha!}\frac{\pa^{|\alpha|}\wi r}{\pa x^\alpha}(f(\zeta))x^\alpha$$ convergent to $\wi r(f(\zeta)+x)$, provided that $x=(x_1,\ldots,x_{2n})\in P_n(0,\eps)$ (where $\NN_0:=\NN\cup\{0\}$ and $|\alpha|:=\alpha_1+\ldots+\alpha_{2n}$). Hence \begin{equation}\label{69}T(r+\varrho,f+h)=\sum_{\alpha\in\NN_0^{2n}}\frac{1}{\alpha!}\left(\frac{\pa^{|\alpha|}r}{\pa x^\alpha}\circ f\right)h^\alpha+\sum_{\alpha\in\NN_0^{2n}}\frac{1}{\alpha!}\left(\frac{\pa^{|\alpha|}\varrho}{\pa x^\alpha}\circ f\right)h^\alpha\end{equation} pointwise for $\varrho\in X_0$ and $h\in W^{2,2}(\TT,\CC^n)$ with $\|h\|_{\sup}<\eps$.

Put $P:=\bigcup_{\zeta\in \TT} P_{2n}(f(\zeta),\eps)$ and for $\wi r\in X_0$ put $||\wi r||_P:=\sup_P|\wi r|$. Let $\wi r$ be equal to $r$ or to $\varrho$, where $\varrho$ lies is in a neighborhood of $0$ in $X_0$. The Cauchy inequalities give
\begin{equation}\label{series}\left|\frac{\pa^{|\alpha|}\wi r}{\pa x^\alpha}(f(\zeta))\right|\leq\frac{\alpha!\|\wi r\|_{P}}{\eps^{|\alpha|}},\quad\zeta\in\TT.\end{equation}
Therefore, $$\left|\left|\frac{\pa^{|\alpha|}\wi r}{\pa x^\alpha}\circ f\right|\right|_W\leq C_1\frac{\alpha!\|\wi r\|_{P}}{\eps^{|\alpha|}}$$ for some $C_1>0$.

There is $C_2>0$ such that $$\|gh^\alpha\|_W\leq C_2^{|\alpha|+1}\|g\|_W\|h_1\|^{\alpha_1}_W\cdotp\ldots\cdotp\|h_{2n}\|^{\alpha_{2n}}_W$$ for $g\in W^{2,2}(\TT,\CC)$, $h\in W^{2,2}(\TT,\CC^n)$, $\alpha\in\NN_0^{2n}$ (see Appendix for a proof of this fact). Using the above inequalities we infer that $$\sum_{\alpha\in\NN_0^{2n}}\left|\left|\frac{1}{\alpha!}\left(\frac{\pa^{|\alpha|}\wi r}{\pa x^\alpha}\circ f\right)h^\alpha\right|\right|_W$$ is convergent if $h$ is small enough on the norm $\|\cdot\|_W$. Therefore, the series~\eqref{69} is absolutely convergent in the norm $\|\cdot\|_W$, whence $T$ is real analytic.

To show the existence of $V_0$ and $\Upsilon$ we will make use of the Implicit Function Theorem. More precisely, we shall show that the partial derivative $$\Xi_{(f,q,\lambda)}(r_0,f_0'(0),f_0,0,1):B\times Q_0\times\mb{R}\longrightarrow Q\times{B^*}\times\mb{C}^n$$ is an isomorphism.
Observe that for any $(\widetilde{f},\widetilde{q},\widetilde{\lambda})\in B\times Q_0\times\mb{R}$ the following equality holds
\begin{multline*}\Xi_{(f,q,\lambda)}(r_0,f_0'(0),f_0,0,1)(\widetilde{f},\widetilde{q},\widetilde{\lambda})=\left.\frac{d}{dt}
\Xi(r_0,f_0'(0),f_0+t\widetilde{f},t\widetilde{q},1+t\widetilde{\lambda})\right|_{t=0}=\\
=((r_{0z}\circ f_0)\widetilde{f}+(r_{0\overline{z}}\circ f_0)\overline{\widetilde{f}},\pi(\zeta\widetilde{q}r_{0z}\circ f_0+\zeta(r_{0zz} \circ
f_0)\widetilde{f}+\zeta(r_{0z\overline{z}}\circ f_0)\overline{\widetilde{f}}),\widetilde{f}'(0)-\widetilde{\lambda}f_0'(0)),
\end{multline*}
where we treat ${r_0}_z,{r_0}_{\overline{z}}$ as row vectors, $\widetilde{f},\overline{\widetilde{f}}$ as column vectors and $r_{0zz}=\left[\frac{\partial^2r_0}{\partial z_j\partial z_k}\right]_{j,k=1}^n$, $r_{0z\overline{z}}=\left[\frac{\partial^2r_0}{\partial z_j\partial\overline z_k}\right]_{j,k=1}^n$ as $n\times n$ matrices.

By the Bounded Inverse Theorem it suffices to show that $\Xi_{(f,q,\lambda)}(r_0,f_0'(0),f_0,0,1)$ is bijective, i.e. for $(\eta,\varphi,v)\in Q\times B^*\times\mb{C}^n$ there exists exactly one $(\widetilde{f},\widetilde{q},\widetilde{\lambda})\in B\times Q_0\times\mb{R}$ satisfying
\begin{equation}
(r_{0z}\circ f_0)\widetilde{f}+(r_{0\overline{z}}\circ f_0)\overline{\widetilde{f}}=\eta,
\label{al1}
\end{equation}
\begin{equation}
\pi(\zeta\widetilde{q}r_{0z}\circ f_0+\zeta (r_{0zz}\circ f_0)\widetilde{f}+\zeta(r_{0z\overline{z}}\circ f_0)\overline{\widetilde{f}})=\varphi,
\label{al2}
\end{equation}
\begin{equation}
\widetilde{f}'(0)-\widetilde{\lambda} f_0'(0)=v.
\label{al3}
\end{equation}
First we show that $\wi\lambda$ and $\wi f_1$ are uniquely determined. Observe that, in view of assumptions, (\ref{al1}) is just $$\frac{1}{2}\overline{\zeta}\widetilde{f}_1+\frac{1}{2}\zeta\overline{\widetilde{f}_1}=\eta$$ or equivalently
\begin{equation}
\re(\widetilde{f}_1/\zeta)=\eta\text{ (on }\TT).
\label{al4}
\end{equation}
Note that the equation (\ref{al4}) uniquely determines $\widetilde{f}_1/\zeta\in W^{2,2}(\TT,\CC)\cap\OO(\DD)\cap\cC(\CDD)$ up to an imaginary additive constant, which may be computed using (\ref{al3}). Actually, $\eta=\re G$ on $\TT$ for some function $G\in W^{2,2}(\TT,\CC)\cap\OO(\DD)\cap\cC(\CDD)$. To see this, let us expand $\eta(\zeta)=\sum_{k=-\infty}^{\infty}a_k\zeta^{k}$, $\zeta\in\TT$. From the equality $\eta(\zeta)=\ov{\eta(\zeta)}$, $\zeta\in\TT$, we get \begin{equation}\label{65}\sum_{k=-\infty}^{\infty}a_k\zeta^{k}=\sum_{k=-\infty}^{\infty}\ov a_k\zeta^{-k}=\sum_{k=-\infty}^{\infty}\ov a_{-k}\zeta^{k},\ \zeta\in\TT,\end{equation} so $a_{-k}=\ov a_k$, $k\in\ZZ$. Hence $$\eta(\zeta)=a_0+\sum_{k=1}^\infty 2\re(a_k\zeta^k)=\re\left(a_0+2\sum_{k=1}^\infty a_k\zeta^k\right),\ \zeta\in\TT.$$ Set $$G(\zeta):=a_0+2\sum_{k=1}^\infty a_k\zeta^k,\ \zeta\in\DD.$$ This series is convergent for $\zeta\in\DD$, so $G\in\OO(\DD)$. Further, the function $G$ extends continuously on $\CDD$ (to the function denoted by the same letter) and the extension lies in $W^{2,2}(\TT,\CC)$. Clearly, $\eta=\re G$ on $\TT$.

We are searching $C\in\RR$ such that the functions $\widetilde{f}_1:=\zeta(G+iC)$ and $\theta:=\im(\widetilde{f}_1/\zeta)$ satisfy $$\eta(0)+i\theta(0)=\widetilde{f}_1'(0)$$ and
$$\eta(0)+i\theta(0)-\widetilde{\lambda}\re{f_{01}'(0)}-i\widetilde{\lambda}\im{{f_{01}'(0)}}=\re{v_1}+i\im{v_1}.$$ But $$\eta(0)-\widetilde{\lambda}\re{f_{01}'(0)}=\re{v_1},$$ which yields $\widetilde{\lambda}$ and then $\theta(0)$, consequently the number $C$.
Having $\widetilde{\lambda}$ and once again using (\ref{al3}), we find uniquely determined $\widetilde{f}_2'(0),\ldots,\widetilde{f}_n'(0)$.

Therefore, the equations $\eqref{al1}$ and $\eqref{al3}$ are satisfied by uniquely determined $\wi f_1$, $\wi\lambda$ and $\widetilde{f}_2'(0),\ldots,\widetilde{f}_n'(0)$.

Consider (\ref{al2}), which is the system of $n$ equations with unknown $\widetilde{q},\widetilde{f}_2,\ldots,\widetilde{f}_n$. Observe that $\widetilde{q}$ appears only in the first of the equations and the remaining $n-1$ equations mean exactly that the mapping
\begin{equation}
\zeta(r_{0\widehat{z}\widehat{z}}\circ f_0)
\widehat{\widetilde{f}}+\zeta(r_{0\widehat{z}\widehat{\overline{z}}}\circ f_0)\widehat{\overline{\widetilde{f}}}-\psi
\label{al5}
\end{equation}
extends holomorphically on $\mb{D}$, where $\widehat{a}:=(a_{2},\ldots,a_{n})$ and $\psi\in W^{2,2}(\TT,\mb{C}^{n-1})$ may be obtained from $\varphi$ and $\widetilde{f}_1$. Indeed, to see this, write (\ref{al2}) in the form $$\pi(F_{1}+\zeta F_{2}+\zeta F_{3})=(\phi_1,\ldots,\phi_n),$$ where $$F_1:=(\wi q,0,\ldots,0),$$$$F_2:=(A_{j})_{j=1}^n,\ A_{j}:=\sum\limits_{k=1}^n(r_{0z_jz_k}\circ f_0)\widetilde{f}_k,$$$$F_3=(B_{j})_{j=1}^n,\ B_{j}:=\sum\limits_{k=1}^n(r_{0z_j\ov z_k}\circ f_0)\overline{\widetilde{f}_k}.$$ It follows that $$\widetilde{q}+\zeta A_1+\zeta B_1-\phi_1$$ and $$\zeta A_j+\zeta B_j-\phi_j,\ j=2,\ldots,n,$$ extend holomorphically on $\mb{D}$ and $$\psi:=\left(\phi_j-\zeta(r_{0z_jz_1}\circ f_0)\widetilde{f}_1-\zeta(r_{0z_j\ov z_1}\circ f_0)\overline{\widetilde{f}_1}\right)_{j=2}^n.$$
Put $$g(\zeta):=\widehat{\widetilde{f}}(\zeta)/\zeta,\quad\alpha(\zeta):=\zeta^2r_{0\widehat{z}\widehat{z}}(f_0(\zeta)),
\quad\beta(\zeta):=r_{0\widehat{z}\widehat{\overline{z}}}(f_0(\zeta)).$$

Observe that $\alpha(\zeta)$, $\beta(\zeta)$ are the $(n-1)\times(n-1)$ matrices depending real analytically on $\zeta$ and $g(\zeta)$ is a column vector in $\mb{C}^{n-1}$. This allows us to reduce \eqref{al5} to the following problem: we have to find a unique $g\in W^{2,2}(\TT,\mb{C}^{n-1})\cap\OO(\DD)\cap\cC(\CDD)$ such that \begin{equation}
\alpha g+\beta\overline{g}-\psi\text{ extends holomorphically on $\mb{D}$ and } g(0)={\widehat{\widetilde{f}'}}(0).
\label{al6}
\end{equation}
The fact that every $f_0(\zeta)$ is a point of strong linear convexity of the domain $D_0$ may be written as
\begin{equation}
|X^T\alpha(\zeta)X|<X^{T}\beta(\zeta)\overline{X},\ \zeta\in\TT,\ X\in(\mb{C}^{n-1})_*.
\label{al7}
\end{equation}

Note that $\beta(\zeta)$ is self-adjoint and strictly positive, hence using Proposition \ref{12} we get a mapping $H\in\OO(\CDD,\CC^{(n-1)\times(n-1)})$ such that $\det H\neq 0$ on $\CDD$ and $HH^*=\beta$ on $\TT$. Using this notation, (\ref{al6}) is
equivalent to
\begin{equation}
H^{-1}\alpha g+H^*\overline{g}-H^{-1}\psi\text{ extends holomorphically on $\mb{D}$}
\label{al8}
\end{equation}
or, if we denote $h:=H^Tg$, $\gamma:=H^{-1}\alpha (H^T)^{-1}$, to
\begin{equation}
\gamma h+\overline{h}-H^{-1}\psi\text{ extends holomorphically on $\mb{D}.$}
\label{al9}
\end{equation}

For any $\zeta\in\TT$ the operator norm of the symmetric matrix $\gamma(\zeta)$ is uniformly less than 1. In fact, from (\ref{al7}) for any $X\in\mb{C}^{n-1}$ with $|X|=1$ \begin{multline*}|X^{T}\gamma(\zeta)X|=|X^{T}H(\zeta)^{-1}\alpha(\zeta)(H(\zeta)^T)^{-1}X|<X^TH(\zeta)^{-1}\beta(\zeta)
\overline{(H(\zeta)^T)^{-1}X}=\\=X^TH(\zeta)^{-1}H(\zeta)H(\zeta)^*\overline{(H(\zeta)^T)^{-1}}
\overline{X}=|X|^2=1,\end{multline*} so, by the compactness argument, $|X^{T}\gamma(\zeta)X|\leq 1-\wi\eps$ for some $\wi\eps>0$ independent on $\zeta$ and $X$. Thus $\|\gamma(\zeta)\|\leq 1-\wi\eps$ by Proposition \ref{59}.

We have to prove that there is a unique solution $h\in W^{2,2}(\TT,\CC^{n-1})\cap\OO(\DD)\cap\cC(\CDD)$ of (\ref{al9}) such that $h(0)=a$ with a given $a\in\CC^{n-1}$.

Define the operator $$P:W^{2,2}(\TT,\mb{C}^{n-1})\ni\sum_{k=-\infty}^{\infty}a_k\zeta^{k}\longmapsto\overline{\sum_{k=-\infty}^{-1}a_k\zeta^{k}}\in W^{2,2}(\TT,\mb{C}^{n-1}),$$ where $a_k\in\CC^{n-1}$, $k\in\ZZ$.

We will show that a mapping $h\in
W^{2,2}(\TT,\mb{C}^{n-1})\cap\OO(\DD)\cap\cC(\CDD)$ satisfies (\ref{al9}) and $h(0)=a$ if and only if it is a fixed point of the mapping $$K:W^{2,2}(\TT,\mb{C}^{n-1})\ni h\longmapsto P(H^{-1}\psi-\gamma h)+a\in W^{2,2}(\TT,\mb{C}^{n-1}).$$

Indeed, take $h\in
W^{2,2}(\TT,\mb{C}^{n-1})\cap\OO(\DD)\cap\cC(\CDD)$ and suppose that $h(0)=a$ and $\gamma h+\overline{h}-H^{-1}\psi$ extends holomorphically on $\mb{D}$. Then $$h=a+\sum_{k=1}^{\infty}a_k\zeta^{k},\quad\overline{h}=\overline{a}+\sum_{k=1}^{\infty}\overline a_k\zeta^{-k}=\sum_{k=-\infty}^{-1}\overline a_{-k}\zeta^{k}+\overline{a},$$ $$P(h)=0,\quad P(\overline{h})=\sum_{k=1}^{\infty}a_k\zeta^{k}=h-a$$ and $$P(\gamma h+\overline{h}-H^{-1}\psi)=0,$$ which implies $$P(H^{-1}\psi-\gamma h)=h-a$$ and finally $K(h)=h$. Conversely, suppose that $K(h)=h$. Then $$P(H^{-1}\psi-\gamma h)=h-a=\sum_{k=1}^{\infty}a_k\zeta^{k}+a_1-a,\quad P(h)=0$$ and
$$P(\overline{h})=\sum_{k=1}^{\infty}a_k\zeta^{k}=h-a_1,$$ from which follows that $$P(\gamma h+\overline{h}-H^{-1}\psi)=P(\overline{h})-P(H^{-1}\psi-\gamma h)=a-a_1$$ and $$P(\gamma h+\overline{h}-H^{-1}\psi)=0\text{ iff }a=a_1.$$ Observe that $h(0)=K(h)(0)=P(H^{-1}\psi-\gamma h)(0)+a=a$.

We shall make use of the Banach Fixed Point Theorem. To do this, consider $W^{2,2}(\TT,\CC^{n-1})$ equipped with the following norm $$\|h\|_{\varepsilon}:=\|h\|_L+\varepsilon\|h'\|_L+
\varepsilon^2\|h''\|_L,$$ where $\eps>0$ and $\|\cdot\|_L$ is the $L^2$-norm (it is a Banach space). We will prove that $K$ is a contraction with respect to the norm $\|\cdot\|_{\varepsilon}$ for sufficiently small $\eps$. Indeed, there is $\wi\eps>0$ such that for any $h_1,h_2\in W^{2,2}(\TT,\CC^{n-1})$
\begin{equation}
\|K(h_1)-K(h_2)\|_L=\|P(\gamma(h_2-h_1))\|_L\leq\|\gamma(h_2-h_1)\|_L\leq (1-\wi\eps)\|h_2-h_1\|_L.
\label{al10}
\end{equation}
Moreover,
\begin{multline}
\|K(h_1)'-K(h_2)'\|_L= \|P(\gamma h_2)'-P(\gamma h_1)'\|_L\leq\\
\leq\|(\gamma h_2)'-(\gamma h_1)'\|_L= \|\gamma '(h_2-h_1)+\gamma(h_2'-h_1')\|_L.
\label{al11}
\end{multline} Furthermore,
\begin{equation}
\|K(h_1)''-K(h_2)''\|_L\leq\|\gamma ''(h_2-h_1)\|_L+2\|\gamma '(h_2'-h_1')\|_L+\|\gamma
(h_1''-h_2'')\|_L.\label{al12}
\end{equation}
Using the finiteness of $\|\gamma '\|$, $\|\gamma ''\|$ and putting (\ref{al10}), (\ref{al11}), (\ref{al12}) together we see that there exists $\varepsilon>0$ such that $K$ is a contraction w.r.t. the norm $\|\cdot\|_{\varepsilon}$.

We have found $\widetilde{f}$ and $\widetilde{\lambda}$ satisfying (\ref{al1}), (\ref{al3}) and the last $n-1$ equations from (\ref{al2}) are satisfied. 

It remains to show that there exists a unique $\widetilde{q}\in Q_0$ such that $\widetilde{q}+\zeta A_1+\zeta B_1-\varphi_1$ extends holomorphically on $\mb{D}$.

Comparing the coefficients as in \eqref{65}, we see that if $$\pi(\zeta A_1+\zeta B_1-\varphi_1)=\sum_{k=-\infty}^{-1}a_k\zeta^{k}$$
then $\widetilde{q}$ has to be taken as $$-\sum_{k=-\infty}^{-1}a_k\zeta^{k}-\sum_{k=0}^{\infty}b_k\zeta^{k}$$
with $b_k:=\overline a_{-k}$ for $k\geq 1$ and $b_0\in\RR$ uniquely determined by $\widetilde{q}(1)=0$.\\

Let us show that the proof of the second Lemma follows from the proof of the first one.
Since $\wi\Xi$ is real analytic it suffices to prove that the derivative $$\wi\Xi_{(f,q,\xi)}(r_0,f_0(\xi_0),f_0,0,\xi_0):B\times Q_0\times\RR\longrightarrow Q\times{B^*}\times\mb{C}^n$$ is invertible.
For $(\widetilde{f},\widetilde{q},\widetilde{\xi})\in B\times Q_0\times\RR$ we get
\begin{multline*}
\wi\Xi_{(f,q,\xi)}(r_0,f_0(\xi_0),f_0,0,\xi_0)(\widetilde{f},\widetilde{q},\widetilde{\xi})=\left.\frac{d}{dt}
\wi\Xi(r_0,f_0(\xi_0),f_0+t\widetilde{f},t\widetilde{q},\xi_0+t\widetilde{\xi})\right|_{t=0}=\\
=((r_{0z}\circ f_0)\widetilde{f}+(r_{0\overline{z}}\circ f_0)\overline{\widetilde{f}},
\pi(\zeta\widetilde{q}r_{0z}\circ f_0+\zeta(r_{0zz}\circ f_0)\widetilde{f}+\zeta(r_{0z\overline{z}}\circ f_0)\overline{\widetilde{f}}),\widetilde{f}(\xi_0)+\wi\xi f_0'(\xi_0)).
\end{multline*}
We have to show that for $(\eta,\varphi,w)\in Q\times B^*\times\mb{C}^n$ there exists exactly one $(\widetilde{f},\widetilde{q},\widetilde{\xi})\in B\times Q_0\times\RR$ satisfying
\begin{equation}
(r_{0z}\circ f_0)\widetilde{f}+(r_{0\overline{z}}\circ f_0)\overline{\widetilde{f}}=\eta,
\label{1al1}
\end{equation}
\begin{equation}
\pi(\zeta\widetilde{q}r_{0z}\circ f_0+\zeta (r_{0zz}\circ f_0)\widetilde{f}+\zeta(r_{0z\overline{z}}\circ f_0)\overline{\widetilde{f}})=\varphi,
\label{1al2}
\end{equation}
\begin{equation}
\wi f(\xi_0)+\wi\xi f_0'(\xi_0)=w.
\label{1al3}
\end{equation}
The equation (\ref{1al1}) turns out to be
\begin{equation}
\re(\widetilde{f}_1/\zeta)=\eta\text{ (on }\TT).
\label{1al4}
\end{equation}
The equation above uniquely determines $\widetilde{f}_1/\zeta\in W^{2,2}(\TT,\CC)\cap\OO(\DD)\cap\cC(\CDD)$ up to an imaginary additive constant, which may be computed using (\ref{1al3}). Indeed, there exists $G\in W^{2,2}(\TT,\CC)\cap\OO(\DD)\cap\cC(\CDD)$ such that $\eta=\re G$ on $\TT$. We are searching $C\in\RR$ such that the functions $\widetilde{f}_1:=\zeta(G+iC)$ and $\theta:=\im(\widetilde{f}_1/\zeta)$ satisfy  $$\xi_0\eta(\xi_0)+i\xi_0\theta(\xi_0)=\widetilde{f}_1(\xi_0)$$ and $$\xi_0(\eta(\xi_0)+i\theta(\xi_0))+\widetilde{\xi}\re{f_{01}'(\xi_0)}+i\widetilde{\xi}\im{{f_{01}'(\xi_0)}}=
\re{w_1}+i\im{w_1}.$$ But $$\xi_0\eta(\xi_0)+\widetilde{\xi}\re{f_{01}'(\xi_0)}=\re{w_1},$$ which yields $\widetilde{\xi}$ and then $\theta(\xi_0)$, consequently the number $C$. Having $\widetilde{\xi}$ and once again using (\ref{1al3}), we find uniquely determined $\widetilde{f}_2(\xi_0),\ldots,\widetilde{f}_n(\xi_0)$.

Therefore, the equations $\eqref{1al1}$ and $\eqref{1al3}$ are satisfied by uniquely determined $\wi f_1$, $\wi\xi$ and $\widetilde{f}_2(\xi_0),\ldots,\widetilde{f}_n(\xi_0)$.

In the remaining part of the proof we change the second condition of \eqref{al6} to $$g(\xi_0)={\widehat{\widetilde{f}}}(\xi_0)/\xi_0$$ and we have to prove that there is a unique solution $h\in W^{2,2}(\TT,\CC^{n-1})\cap\OO(\DD)\cap\cC(\CDD)$ of (\ref{al9}) such that $h(\xi_0)=a$ with a given $a\in\CC^{n-1}$. Let $\tau$ be an automorphism of $\DD$ (so it extends holomorphically near $\CDD$), which maps $0$ to $\xi_0$, i.e. $$\tau(\xi):=\frac{\xi_0-\xi}{1-\ov\xi_0\xi},\ \xi\in\DD.$$ Let the maps $P,K$ be as before. Then $h\in W^{2,2}(\TT,\mb{C}^{n-1})\cap\OO(\DD)\cap\cC(\CDD)$ satisfies
(\ref{al9}) and $h(\xi_0)=a$ if and only if $h\circ\tau\in W^{2,2}(\TT,\mb{C}^{n-1})\cap\OO(\DD)\cap\cC(\CDD)$ satisfies (\ref{al9}) and $(h\circ\tau)(0)=a$. We already know that there is exactly one $\wi h\in W^{2,2}(\TT,\mb{C}^{n-1})\cap\OO(\DD)\cap\cC(\CDD)$ satisfying (\ref{al9}) and $\wi h(0)=a$. Setting $h:=\wi h\circ\tau^{-1}$, we get the claim.
\end{proof}

\subsection{Topology in the class of domains with real analytic boundaries}\label{topol}

We introduce a concept of a domain being close to some other domain. Let $D_0\su\mb{C}^n$ be a bounded domain with real analytic boundary. Then there exist a neighborhood $U_0$ of $\partial D_0$ and a real analytic defining function $r_0:U_0\longrightarrow\mb{R}$ such that $\nabla r_0$ does not vanish in $U_0$ and $$D_0\cap U_0=\{z\in U_0:r_0(z)<0\}.$$

\begin{dff}
We say that domains $D$ \textit{tend to} $D_0$ $($or are \textit{close to} $D_0${$)$} if one can choose their defining functions $r\in X_0$ such that $r$ tend to $r_0$ in $X_0$.
\end{dff}

\begin{remm} If $r\in X_0$ is near to $r_0$ with respect to the topology in $X_0$, then $\{z\in U_0:r(z)=0\}$ is a compact real analytic hypersurface which bounds a bounded domain. We denote it by $D^{r}$.

Moreover, if $D^{r_0}$ is strongly linearly convex then a domain $D^r$ is also strongly linearly convex provided that $r$ is near $r_0$.
\end{remm}

\subsection{Statement of the main result of this section}

\begin{remm}\label{f} Assume that $D^r$ is a strongly linearly convex domain bounded by a real analytic hypersurface $\{z\in U_0:r(z)=0\}$. Let $\xi\in(0,1)$ and $w\in(\CC^n)_*$.

Then a function $f\in B_0$ satisfies the conditions $$f\text{ is a weak stationary mapping of }D^r,\ f(0)=0,\ f(\xi)=w$$ if and only if there exists $q\in Q_0$ such that $q>-1$ and $\wi\Xi(r,w,f,q,\xi)=0$.

Actually, from $\wi\Xi(r,w,f,q,\xi)=0$ we deduce immediately that $r\circ f=0$ on $\TT$, $f(\xi)=w$ and $\pi(\zeta(1+q)(r_z\circ f))=0$. From the first equality we get $f(\TT)\subset \partial D^{r}$. From the last one we deduce that the condition (3') of Definition~\ref{21} is satisfied (with $\rho:=(1+q)|r_z\circ f|$). Since $D^{r}$ is strongly linearly convex, $\ov{D^r}$ is polynomially convex (use the fact that projections of $\CC$-convex domains are $\CC$-convex, as well, and the fact that $D^r$ is smooth). In particular, $$f(\CDD)=f(\widehat{\TT})\subset\widehat{f(\TT)}\subset\widehat{\ov{D^r}}=\ov{D^r},$$ where $\wh S:=\{z\in\CC^m:|P(z)|\leq\sup_S|P|\text{ for any polynomial }P\in\CC[z_1,\ldots,z_m]\}$ is the polynomial hull of a set $S\su\CC^m$. 

Note that this implies $f(\DD)\su D^r$ --- this follows from the fact that $\pa D^r$ does not contain non-constant analytic discs (as $D^r$ is strongly pseudoconvex).

The opposite implication is clear.

\bigskip

In a similar way we show that for any $v\in(\CC^n)_*$ and $\lambda>0$, a function $f\in B_0$ satisfies the conditions $$f\text{ is a weak stationary mapping of }D^r,\ f(0)=0,\ f'(0)=\lambda v$$ if and only if there exists $q\in Q_0$ such that $q>-1$ and $\Xi(r,v,f,q,\lambda)=0$.

\end{remm}

\begin{propp}\label{13} Let $D_0\su\CC^n$, $n\geq 2$, be a strongly linearly convex domain with real analytic boundary and let $f_0:\DD\longrightarrow D_0$ be an $E$-mapping.

$(1)$ Let $\xi_0\in(0,1)$. Then there exist a neighborhood $W_0$ of $(r_0,f_0(\xi_0))$ in $X_0\times D_0$ and real analytic mappings $$\Lambda:W_0\longrightarrow\mc{C}^{1/2}(\overline{\mb{D}}),\ \Omega:W_0\longrightarrow(0,1)$$ such that $$\Lambda(r_0,f_0(\xi_0))=f_0,\ \Omega(r_0,f_0(\xi_0))=\xi_0$$ and for any $(r,w)\in W_0$ the mapping
$f:=\Lambda(r,w)$ is an $E$-mapping of $D^{r}$ satisfying $$f(0)=f_0(0)\text{ and }f(\Omega(r,w))=w.$$

$(2)$ There exist a neighborhood $V_0$ of $(r_0,f_0'(0))$ in $X_0\times\mb{C}^n$
and a real analytic mapping $$\Gamma:V_0\longrightarrow\mc{C}^{1/2}(\overline{\mb{D}})$$ such that $$\Gamma(r_0,f_0'(0))=f_0$$ and for any $(r,v)\in V_0$ the mapping $f:=\Gamma(r,v)$ is an $E$-mapping of $D^{r}$ satisfying $$f(0)=f_0(0)\text{ and }f'(0)=\lambda v\text{ for some }\lambda>0.$$
\end{propp}

\begin{proof}

Observe that Proposition \ref{11} provides us with a mapping $g_0=\Phi\circ f_0$ and a domain $G_0:=\Phi(D_0)$ giving a data for situation (\dag) (here $\partial D_0$ is contained in $U_0$). Clearly, $\rho_0:=r_0\circ\Phi^{-1}$ is a defining function of $G_0$.

Using Lemmas \ref{cruciallemma}, \ref{cruciallemma1} we get neighborhoods $V_0$, $W_0$ of $(\rho_0, g_0'(0))$, $(\rho_0,g_0(\xi_0))$ respectively and real analytic mappings $\Upsilon$, $\wi\Upsilon$ such that $ \Xi(\rho,v,\Upsilon(\rho,v))=0$ on $V_0$ and $ \wi\Xi(\rho,w,\wi\Upsilon(\rho,w))=0$ on $W_0$. Define $$\wh\Lambda:=\pi_B\circ\wi\Upsilon,\quad\Omega:=\pi_\RR\circ\wi\Upsilon,\quad\wh\Gamma:=\pi_B\circ\Upsilon,$$ where $$\pi_B:B\times Q_0\times\mb{R}\longrightarrow B,\quad\pi_\RR:B\times Q_0\times\mb{R}\longrightarrow\RR,\ $$ are the projections.

If $\rho$ is sufficiently close to $\rho_0$, then the hypersurface $\{\rho=0\}$ bounds a strongly linearly convex domain. Moreover, then $\wh\Lambda(\rho,w)$ and $\wh\Gamma(\rho,v)$ are extremal mappings in $G^{\rho}$ (see Remark~\ref{f}).

Composing $\wh\Lambda(\rho,w)$ and $\wh\Gamma(\rho,v)$ with $\Phi^{-1}$ and making use of Remark \ref{rem:theta} we get weak stationary mappings in $D^r$, where $r:=\rho\circ\Phi$. To show that they are $E$-mappings we proceed as follows. If $D^r$ is sufficiently close to $D_0$ (this depends on a distance between $\rho$ and $\rho_0$), the domain $D^r$ is strongly linearly convex, so by the results of Section \ref{55} $$\Lambda(r,w):=\Phi^{-1}\circ\wh\Lambda(\rho,w)\text{\  and\  }\Gamma(r,v):=\Phi^{-1}\circ\wh\Gamma(\rho,v)$$ are stationary mappings. Moreover, they are close to $f_0$ provided that $r$ is sufficiently close to $r_0$. Therefore, their winding numbers are equal. Thus $f$ satisfies condition (4) of Definition~\ref{21e}, i.e. $f$ is an $E$-mapping.
\end{proof}

\section{Localization property}

\begin{prop}\label{localization} Let $D\su\mathbb C^n$, $n\geq 2$, be a domain. Assume that $a\in\partial D$ is such that $\partial D$ is real analytic and strongly convex in a neighborhood of $a$. Then for any sufficiently small neighborhood $V_0$ of $a$ there is a weak stationary mapping of $D\cap V_0$ such that $f(\mathbb T)\su\partial D$.

In particular, $f$ is a weak stationary mapping of $D$.
\end{prop}

\begin{proof} Let $r$ be a real analytic defining function in a neighborhood of $a$. The problem we are dealing with has a local character, so replacing $r$ with $r\circ\Psi$, where $\Psi$ is a local biholomorphism near $a$, we may assume that $a=(0,\ldots,0,1)$ and a defining function of $D$ near $a$ is $r(z)=-1+|z|^2+h(z-a)$, where $h$ is real analytic in a neighborhood of $0$ and $h(z)=O(|z|^3)$ as $z\to 0$ (cf. \cite{Rud}, p. 321).

Following \cite{Lem2}, let us consider the mappings
$$A_t(z):=\left((1-t^2)^{1/2}\frac{z'}{1+tz_n},\frac{z_n+t}{1+tz_n}\right),\quad z=(z',z_n)\in\CC^{n-1}\times\DD,\,\,t\in(0,1),$$ which restricted to $\BB_n$ are automorphisms. Let $$r_t(z):=\begin{cases}\frac{|1+tz_n|^2}{1-t^2}r(A_t(z)),&t\in(0,1),\\-1+|z|^2,&t=1.\end{cases}$$ It is clear that $f_{(1)}(\zeta)=(\zeta,0,\ldots,0)$, $\zeta\in\DD$ is a stationary mapping of $\mathbb B_n$. We want to have the situation (\dag) which will allow us to use Lemma \ref{cruciallemma} (or Lemma \ref{cruciallemma1}). Note that $r_t$ does not converge to $r_1$ as $t\to 1$. However, $r_t\to r_1$ in $X_0(U_0,U_0^{\mathbb C})$, where $U_0$ is a neighborhood of $f_{(1)}(\TT)$ contained in $\{z\in\mathbb C^n:\re z_n>-1/2\}$ and $U_0^{\mathbb C}$ is sufficiently small (remember that $h(z)=O(|z|^3)$).

Therefore, making use of Lemma \ref{cruciallemma} for $t$ sufficiently close to $1$ we obtain stationary mappings $f_{(t)}$ in $D_t:=\{z\in \mathbb C^n: r_t(z)<0,\ \re z_n>-1/2\}$ such that $f_{(t)}\to f_{(1)}$ in the $W^{2,2}$-norm (so also in the sup-norm). Actually, it follows from Lemma~\ref{cruciallemma} that one may take $f_{(t)}:=\pi_B\circ\Upsilon(r_t,f_{(1)}'(0))$ (keeping the notation from this lemma). The argument used in Remark~\ref{f} gives that $f_{(t)}$ satisfies conditions (1'), (2') and (3') of Definition~\ref{21}. Since the non-constant function $r\circ A_t\circ f_{(t)}$ is subharmonic on $\DD$, continuous on $\CDD$ and $r\circ A_t\circ f_{(t)}=0$ on $\TT$, we see from the maximum principle that $f_{(t)}$ maps $\DD$ in $D_t$. Therefore, $f_{(t)}$ are weak stationary mappings for $t$ close to $1$.

In particular, $$f_{(t)}(\DD)\subset 2\mathbb B_n \cap \{z\in\mathbb C^n:\re z_n>-1/2\}$$ provided that $t$ is close to $1$. The mappings $A_t$ have the following important property $$A_t(2\mathbb B_n\cap\{z\in\mathbb C^n:\re z_n>-1/2\})\to\{a\}$$ as $t\to 1$ in the sense of the Hausdorff distance.

Therefore, we find from Remark \ref{rem:theta} that $g_{(t)}:=A_t\circ f_{(t)}$ is a stationary mapping of $D$. Since $g_{(t)}$ maps $\DD$ onto arbitrarily small neighborhood of $a$ provided that $t$ is sufficiently close to $1$, we immediately get the assertion.
\end{proof}

\section{Proofs of Theorems \ref{lem-car} and \ref{main}}

We start this section with the following
\begin{lem}\label{lemat} For any different $z,w\in D$ $($resp. for any $z\in D$, $v\in(\CC^n)_*${$)$} there exists an $E$-mapping $f:\DD\longrightarrow D$ such that $f(0)=z$, $f(\xi)=w$ for some $\xi\in(0,1)$ $($resp. $f(0)=z$, $f'(0)=\lambda v$ for some $\lambda>0${$)$}.
\end{lem}

\begin{proof}
Fix different $z,w\in D$ (resp. $z\in D$, $v\in(\CC^{n})_*$).

First, consider the case when $D$ is bounded strongly convex with real analytic boundary. Without loss of generality one may assume that $0\in D\Subset\BB_n$. We need some properties of the Minkowski functionals.

Let $\mu_G$ be a Minkowski functional of a domain $G\subset\CC^n$ containing the origin, i.e. $$\mu_G(x):=\inf\left\{s>0:\frac{x}{s}\in G\right\},\ x\in\CC^n.$$ Assume that $G$ is bounded strongly convex with real analytic boundary. We shall show that
\begin{itemize}
\item $\mu_G-1$ is a real analytic outside $0$, defining function of $G$;
\item $\mu^2_G-1$ is a real analytic outside $0$, strongly convex outside $0$, defining function of $G$.
\end{itemize}
Clearly, $G=\{x\in\RR^{2n}:\mu_G(x)<1\}$. Setting $$q(x,s):=r\left(\frac{x}{s}\right),\ (x,s)\in U_0\times U_1,$$ where $r$ is a real analytic defining function of $G$ (defined near $\pa G$) and $U_0\su\RR^{2n}$, $U_1\su\RR$ are neighborhoods of $\pa G$ and $1$ respectively, we have $$\frac{\partial q}{\partial s}(x,s)=-\frac{1}{s^2}\left\langle\nabla r\left(\frac{x}{s}\right),x\right\rangle_{\RR}\neq 0$$ for $(x,s)$ such that $x\in\partial G$ and $s=\mu_G(x)=1$ (since $0\in G$, the vector $-x$ hooked at the point $x$ is inward $G$, so it is not orthogonal to the normal vector at $x$). By the Implicit Function Theorem for the equation $q=0$, the function $\mu_G$ is real analytic in a neighborhood $V_0$ of $\partial G$. To see that $\mu_G$ is real analytic outside $0$, fix $x_0\in(\RR^{2n})_*$. Then the set $$W_0:=\left\{x\in\RR^{2n}:\frac{x}{\mu_G(x_0)}\in V_0\right\}$$ is open and contains $x_0$. Since $$\mu_G(x)=\mu_G(x_0)\mu_G\left(\frac{x}{\mu_G(x_0)}\right),\ x\in W_0,$$ the function $\mu_G$ is real analytic in $W_0$. Therefore, we can take $d/ds$ on both sides of $\mu_G(sx)=s\mu_G(x),\ x\neq 0,\ s>0$ to obtain $$\langle\nabla\mu_G(x),x\rangle_{\RR}=\mu_G(x),\ x\neq 0,$$ so $\nabla\mu_G\neq 0$ in $(\RR^{2n})_*$.

Furthermore, $\nabla\mu^2_G=2\mu_G\nabla\mu_G$, so $\mu^2_G-1$ is also a defining function of $G$.
To show that $u:=\mu^2_G$ is strongly convex outside $0$ let us prove that $$X^T\mathcal{H}_aX>0,\quad a\in\pa G,\ X\in(\RR^{2n})_*,$$ where $\mathcal{H}_x:=\mathcal{H}u(x)$ for $x\in(\RR^{2n})_*$. Taking $\pa/\pa x_j$ on both sides of $$u(sx)=s^2u(x),\ x,s\neq 0,$$ we get \begin{equation}\label{62}\frac{\pa u}{\pa x_j}(sx)=s\frac{\pa u}{\pa x_j}(x)\end{equation} and further taking $d/ds$ $$\sum_{k=1}^{2n}\frac{\pa^2 u}{\pa x_j\pa x_k}(sx)x_k=\frac{\pa u}{\pa x_j}(x).$$ In particular, $$x^T\mathcal{H}_xy=\sum_{j,k=1}^{2n}\frac{\pa^2 u}{\pa x_k\pa x_j}(x)x_ky_j=\langle\nabla u(x),y\rangle_{\RR},\ x\in(\RR^{2n})_*,\ y\in\RR^{2n}.$$ Let $a\in\pa G$. Since $\langle\nabla\mu_G(a),a\rangle_{\RR}=\mu_G(a)=1$, we have $a\notin T^\RR_G(a)$. Any $X\in(\RR^{2n})_*$ can be represented as $\alpha a+\beta Y$, where $Y\in T^\RR_G(a)$, $\alpha,\beta\in\RR$, $(\alpha,\beta)\neq(0,0)$. Then \begin{eqnarray*}X^T\mathcal{H}_aX&=&\alpha^2a^T\mathcal{H}_aa+2\alpha\beta a^T\mathcal{H}_aY+\beta^2Y^T\mathcal{H}_aY=\\&=&\alpha^2\langle\nabla u(a),a\rangle_{\RR} +2\alpha\beta\langle\nabla u(a),Y\rangle_{\RR} +\beta^2Y^T\mathcal{H}_aY= \\&=&\alpha^22\mu_G(a)\langle\nabla\mu_G(a),a\rangle_{\RR} +\beta^2Y^T\mathcal{H}_aY=
2\alpha^2+\beta^2Y^T\mathcal{H}_aY.\end{eqnarray*} Since $G$ is strongly convex, the Hessian of any defining function is strictly positive on the tangent space, i.e. $Y^T\mathcal{H}_aY>0$ if $Y\in(T^\RR_G(a))_*$. Hence $X^T\mathcal{H}_aX\geq 0$. Note that it cannot be $X^T\mathcal{H}_aX=0$, since then $\alpha=0$, consequently $\beta\neq 0$ and $Y^T\mathcal{H}_aY=0$. On the other side $Y=X/\beta\neq 0$ --- a contradiction.

Taking $\pa/\pa x_k$ on both sides of \eqref{62} we obtain $$\frac{\pa^2 u}{\pa x_j\pa x_k}(sx)=\frac{\pa^2 u}{\pa x_j\pa x_k}(x),\ x,s\neq 0$$ and for $a,X\in(\RR^{2n})_*$ $$X^T\mathcal{H}_aX=X^T\mathcal{H}_{a/\mu_G(a)}X>0.$$

Let us consider the sets $$D_t:=\{x\in\CC^n:t\mu^2_D(x)+(1-t)\mu^2_{\BB_n}(x)<1\},\ t\in[0,1].$$ The functions $t\mu^2_D+(1-t)\mu^2_{\BB_n}$ are real analytic in $(\CC^n)_*$ and strongly convex in $(\CC^n)_*$, so $D_t$ are strongly convex domains with real analytic boundaries satisfying $$D=D_1\Subset D_{t_2}\Subset D_{t_1}\Subset D_0=\BB_n\text{\  if \ }0<t_1<t_2<1.$$ It is clear that $\mu_{D_t}=\sqrt{t\mu^2_D+(1-t)\mu^2_{\BB_n}}$. Further, if $t_{1}$ is close to $t_{2}$ then $D_{t_{1}}$ is close to $D_{t_{2}}$ w.r.t. the topology introduced in Section \ref{27}. We want to show that $D_t$ are in some family $\mathcal D(c)$. Only the interior and exterior ball conditions need to verify.

There exists $\delta>0$ such that $\delta\BB_n\Subset D$. Further, $\nabla\mu_{D_t}^2\neq 0$ in $(\RR^{2n})_*$. Set $$M:=\sup\left\{\frac{\mathcal{H}\mu_{D_t}^2(x;X)}{|\nabla\mu_{D_t}^2(y)|}:
t\in[0,1],\ x,y\in 2\ov{\BB}_n\setminus\delta\BB_n,\ X\in\RR^{2n},\ |X|=1\right\}.$$ It is a positive number since the functions $\mu_{D_t}^2$ are strongly convex in $(\RR^{2n})_*$ and the `sup' of the continuous, positive function is taken over a compact set. Let $$r:=\min\left\{\frac{1}{2M},\frac{\dist(\pa D,\delta\BB_n)}{2}\right\}.$$ For fixed $t\in[0,1]$ and $a\in\pa D_t$ put $a':=a-r\nu_{D_t}(a)$. In particular, $\ov{B_n(a',r)}\su 2\ov{\BB}_n\setminus\delta\BB_n$. Let us define $$h(x):=\mu^2_{D_t}(x)-\frac{|\nabla\mu^2_{D_t}(a)|}{2|a-a'|}(|x-a'|^2-r^2),\ x\in 2\ov{\BB}_n\setminus\delta\BB_n.$$ We have $h(a)=1$ and $$\nabla h(x)=\nabla\mu^2_{D_t}(x)-\frac{|\nabla\mu^2_{D_t}(a)|}{|a-a'|}(x-a').$$ For $x=a$, dividing the right side by $|\nabla\mu^2_{D_t}(a)|$, we get a difference of the same normal vectors $\nu_{D_t}(a)$, so $\nabla h(a)=0$. Moreover, for $|X|=1$ $$\mathcal{H}h(x;X)=\mathcal{H}\mu^2_{D_t}(x;X)-\frac{|\nabla\mu^2_{D_t}(a)|}{r}\leq M|\nabla\mu^2_{D_t}(a)|-2M|\nabla\mu^2_{D_t}(a)|<0.$$ It follows that $h\leq 1$ in any convex set $S$ such that $a\in S\su 2\ov{\BB}_n\setminus\delta\BB_n$. Indeed, assume the contrary. Then there is $y\in S$ such that $h(y)>1$. Let us join $a$ and $y$ with an interval $$g:[0,1]\ni t\longmapsto h(ta+(1-t)y)\in S.$$ Since $a$ is a strong local maximum of $h$, the function $g$ has a local minimum at some point $t_0\in(0,1)$. Hence $$0\leq g''(t_0)=\mathcal{H}h(t_0a+(1-t_0)y;a-y),$$ which is impossible.

Setting $S:=\ov{B_n(a',r)}$, we get $$\mu^2_{D_t}(x)\leq 1+\frac{|\nabla\mu^2_{D_t}(a)|}{2|a-a'|}(|x-a'|^2-r^2)<1$$  for $x\in B_n(a',r)$, i.e. $x\in D_t$.

The proof of the exterior ball condition is similar. Set $$m:=\inf\left\{\frac{\mathcal{H}\mu_{D_t}^2(x;X)}{|\nabla\mu_{D_t}^2(y)|}:
t\in[0,1],\ x,y\in(\ov{\BB}_n)_*,\ X\in\RR^{2n},\ |X|=1\right\}.$$ Note that the $m>0$. Actually, the homogeneity of $\mu_{D_t}$ implies $\mathcal{H}\mu_{D_t}^2(sx;X)=\mathcal{H}\mu_{D_t}^2(x;X)$ and $\nabla\mu_{D_t}^2(sx)=s\nabla\mu_{D_t}^2(x)$ for $x\neq 0$, $X\in \RR^{2n}$, $s>0$. Therefore, there are positive constants $C_1,C_2$ such that $C_1\leq\mathcal{H}\mu_{D_t}^2(x;X)$ for $x\neq 0$, $X\in \RR^{2n}$, $|X|=1$ and $|\nabla\mu_{D_t}^2(y)|\leq C_2$ for $y\in\ov\BB_n$. In particular, $m\geq C_1/C_2$.

Let $R:=2/m$. For fixed $t\in[0,1]$ and $a\in\pa D_t$ put $a'':=a-R\nu_{D_t}(a)$. Let us define $$\wi h(x):=\mu^2_{D_t}(x)-\frac{|\nabla\mu^2_{D_t}(a)|}{2|a-a''|}(|x-a''|^2-R^2),\ x\in\ov{\BB}_n.$$ We have $\wi h(a)=1$ and $$\nabla\wi  h(x)=\nabla\mu^2_{D_t}(x)-\frac{|\nabla\mu^2_{D_t}(a)|}{|a-a''|}(x-a''),$$ so $\nabla\wi h(a)=0$. Moreover, for $x\in(\ov{\BB}_n)_*$ and $|X|=1$ $$\mathcal{H}\wi h(x;X)=\mathcal{H}\mu^2_{D_t}(x;X)-\frac{|\nabla\mu^2_{D_t}(a)|}{R}\geq m|\nabla\mu^2_{D_t}(a)|-m/2|\nabla\mu^2_{D_t}(a)|>0.$$ Therefore, $a$ is a strong local minimum of $\wi h$.

Now using the properties listed above we may deduce that $\wi h\geq 1$ in $\ov\BB_n$. We proceed similarly as before: seeking a contradiction suppose that there is $y\in\ov\BB_n$ such that $\wi h(y)<1$. Moving $y$ a little (if necessary) we may assume that $0$ does not lie on the interval joining $a$ and $y$. Then the mapping $\wi g(t):=\wi h(ta+ (1-t)y)$ attains its local maximum at some point $t_0\in(0,1)$. The second derivative of $\wi g$ at $t_0$ is non-positive, which gives a contradiction with a positivity of the Hessian of the function $\wi h$. 

Hence, we get $$\frac{|\nabla\mu^2_{D_t}(a)|}{2|a-a''|}(|x-a''|^2-R^2)\leq\mu^2_{D_t}(x)-1<0,$$  for $x\in D_t$, so $D_t \subset B_n(a'',R)$.

Let $T$ be the set of all $t\in[0,1]$ such that there is an $E$-mapping $f_{t}:\DD\longrightarrow D_{t}$ with $f_{t}(0)=z$, $f_{t}(\xi_{t})=w$ for some $\xi_{t}\in(0,1)$ (resp. $f_{t}(0)=z$, $f_{t}'(0)=\lambda_{t}v$ for some $\lambda_{t}>0$). We claim that $T=[0,1]$. To prove it we will use the open-close argument.

Clearly, $T\neq\emptyset$, as $0\in T$. Moreover, $T$ is open in $[0,1]$. Indeed, let $t_{0}\in T$. It follows from Proposition \ref{13} that there is a neighborhood $T_{0}$ of $t_{0}$ such that there are $E$-mappings $f_{t}:\DD\longrightarrow D_{t}$ and $\xi_{t}\in(0,1)$ such that $f_{t}(0)=z$, $f_{t}(\xi_{t})=w$ for all $t\in T_{0}$ (resp. $\lambda_{t}>0$ such that $f_{t}(0)=z$, $f_{t}'(0)=\lambda_{t} v$ for all $t\in T_{0}$).

To prove that $T$ is closed, choose a sequence $\{t_{m}\}\su T$ convergent to some $t\in[0,1]$. We want to show that $t\in T$. Since $f_{t_m}$ are $E$-mappings, they are complex geodesics. Therefore, making use of the inclusions $D\subset D_{t_m}\subset\mathbb B_n$ we find that there is a compact set $K\su(0,1)$ (resp. a compact set $\widetilde K\subset(0,\infty)$) such that $\{\xi_{t_m}\}\subset K$ (resp. $\{\lambda_{t_m}\}\subset\widetilde K$). By Propositions \ref{8} and \ref{10b} the functions $f_{t_{m}}$ and $\widetilde f_{t_{m}}$ are equicontinuous in $\mathcal{C}^{1/2}(\overline{\DD})$ and by Propositions \ref{9} and \ref{10a} the functions $\rho_{t_{m}}$ are uniformly bounded from both sides by positive numbers and equicontinuous in $\mathcal{C}^{1/2}(\TT)$. From the Arzela-Ascoli Theorem there are a subsequence $\{s_{m}\}\subset\{t_{m}\}$ and mappings $f,\wi f\in\OO(\DD)\cap\mathcal C^{1/2}(\overline{\mathbb D})$, $\rho\in\cC^{1/2}(\TT)$ such that $f_{s_{m}}\to f$, $\widetilde{f}_{s_{m}}\to\wi f$ uniformly on $\overline{\DD}$, $\rho_{s_{m}}\to\rho$ uniformly on $\TT$ and $\xi_{s_m}\to\xi\in (0,1)$ (resp. $\lambda_{s_m}\to\lambda>0$).

Clearly, $f(\CDD)\su\overline{D}_{t}$, $f(\TT)\su\partial D_{t}$ and $\rho>0$. By the strong pseudoconvexity of $D_t$ we get $f(\DD)\su D_t$.

The conditions (3') and (4) of Definitions~\ref{21} and \ref{21e} follow from the uniform convergence of suitable functions. Therefore, $f$ is a weak $E$-mapping of $D_{t}$, consequently an $E$-mapping of $D_t$, satisfying $f(0)=z$, $f(\xi)=w$ (resp. $f(0)=z$, $f'(0)=\lambda v$).

Let us go back to the general situation that is when a domain $D$ is bounded strongly linearly convex with real analytic boundary. Take a of point $\eta\in\partial{D}$ such that $\max_{\zeta\in\partial{D}}|z-\zeta|=|z-\eta|$. Then $\eta$ is a point of the strong convexity of $D$. Indeed, by the Implicit Function Theorem one can assume that in a neighborhood of $\eta$ the defining functions of $D$ and $B:=B_n(z,|z-\eta|)$ are of the form $r(x):=\wi r(\wi x)-x_{2n}$ and $q(x):=\wi q(\wi x)-x_{2n}$ respectively, where $x=(\wi x,x_{2n})\in\RR^{2n}$ is sufficiently close to $\eta$. From the inclusion $D\su B$ it it follows that $r-q\geq 0$ near $\eta$ and $(r-q)(\eta)=0$. Thus the Hessian $\mathcal{H}(r-q)(\eta)$ is weakly positive in $\CC^n$. Since $\mathcal{H}q(\eta)$ is strictly positive on $T_B^\RR(\eta)_*=T_D^\RR(\eta)_*$, we find that $\mathcal{H}r(\eta)$ is strictly positive on $T_D^\RR(\eta)_*$, as well.

By a continuity argument, there is a convex neighborhood $V_0$ of $\eta$ such that all points from $\pa D\cap V_0$ are points of the strong convexity of $D$. It follows from Proposition \ref{localization} (after shrinking $V_0$ if necessary) that there is a weak stationary mapping $g:\DD\longrightarrow D\cap V_0$ such that $g(\TT)\subset\partial D$. In particular, $g$ is a weak stationary mapping of $D$. Since $D\cap V_0$ is convex, the condition with the winding number is satisfied on $D\cap V_0$ (and then on the whole $D$). Consequently $g$ is an $E$-mapping of $D$.

If $z=g(0)$, $w=g(\xi)$ for some $\xi\in\DD$ (resp. $z=g(0)$, $v=g'(0)$) then there is nothing to prove. In the other case let us take curves $\alpha:[0,1]\longrightarrow D$, $\beta:[0,1]\longrightarrow D$ joining $g(0)$ and $z$, $g(\xi)$ and $w$ (resp. $g(0)$ and $z$, $g'(0)$ and $v$). We may assume that the images of $\alpha$ and $\beta$ are disjoint. Let $T$ be the set of all $t\in[0,1]$ such that there is an $E$-mapping $g_{t}:\DD\longrightarrow D$ such that $g_{t}(0)=\alpha(t)$, $g_{t}(\xi_{t})=\beta(t)$ for some $\xi_{t}\in(0,1)$ (resp. $g_{t}(0)=\alpha(t)$, $g_{t}'(0)=\lambda_{t}\beta(t)$ for some $\lambda_{t}>0$). Again $T\neq\emptyset$ since $0\in T$. Using the results of Section \ref{22} similarly as before (but for one domain), we see that $T$ is closed.

Since $\wi k_D$ is symmetric, it follows from Proposition \ref{13}(1) that the set $T$ is open in $[0,1]$ (first we move along $\alpha$, then by the symmetry we move along $\beta$). Therefore, $g_1$ is the $E$-mapping for $z,w$.

In the case of $\kappa_{D}$ we change a point and then we change a direction. To be more precise, consider the set $S$ of all $s\in[0,1]$ such that there is an $E$-mapping $h_{s}:\DD\longrightarrow D$ such that $h_{s}(0)=\alpha(s)$. Then $0\in S$, by Proposition \ref{13}(1) the set $S$ is open in $[0,1]$ and by results of Section~\ref{22} again, it is closed. Hence $S=[0,1]$. Now we may join $h'_{1}(0)$ and $v$ with a curve $\gamma:[0,1]\longrightarrow \mathbb C^n$. Let us define $R$ as the set of all $r\in[0,1]$ such that there is an $E$-mapping $\wi h_{r}:\DD\longrightarrow D$ such that $\wi h_{r}(0)=h_1(0)$, $\wi h'_{r}(0)=\sigma_{r}\gamma(1-r)$ for some $\sigma_r>0$. Then $1\in R$, by Proposition \ref{13}(2) the set $R$ is open in $[0,1]$ and, by Section \ref{22}, it is closed. Hence $R=[0,1]$, so $\wi h_{0}$ is the $E$-mapping for $z,v$.
\end{proof}

Now we are in position that allows us to prove the main results of the Lempert's paper.

\begin{proof}[Proof of Theorem \ref{lem-car} $($real analytic case$)$] It follows from Lemma \ref{lemat} that for any different points $z,w\in D$ (resp. $z\in D$, $v\in(\CC^n)_*$) one may find an $E$-mapping passing through them (resp. $f(0)=z$, $f'(0)=v$). On the other hand, it follows from Proposition \ref{1} that $E$-mappings have left inverses, so they are complex geodesics.
\end{proof}

\begin{proof}[Proof of Theorem \ref{main} $($real analytic case$)$] This is a direct consequence of Lemma \ref{lemat} and Corollary \ref{28}.
\end{proof}

\bigskip

\begin{center}{\sc $\cC^2$-smooth case}\end{center}
\bigskip

\begin{lem}\label{un} Let $D\su\mathbb C^n$, $n\geq 2$, be a bounded strongly pseudoconvex domain with $\mathcal C^2$-smooth boundary. Take $z\in D$ and let $r$ be a defining function of $D$ such that 
\begin{itemize}\item $r\in \mathcal C^2(\mathbb C^n);$
\item $D=\{x\in \mathbb C^n:r(x)<0\}$;
\item $\mathbb C^n\setminus D=\{x\in \mathbb C^n:r(x)>0\}$;
\item $|\nabla r|=1$ on $\partial D;$
\item $\sum_{j,k=1}^n\frac{\partial^2 r}{\partial z_j\partial\overline z_k}(a)X_{j}\overline{X}_{k}\geq C|X|^2$ for any $a\in \partial D$ and $X\in \mathbb C^n$ with some constant $C>0$.
\end{itemize}

Suppose that there is a sequence $\{r_m\}$ of $\mathcal C^2$-smooth real-valued functions such that $D^{\alpha}r_n$ converges to $D^{\alpha}r$ locally uniformly for any $\alpha\in \mathbb N_0^{2n}$ such that  $|\alpha|:=|\alpha_1| +\ldots+|\alpha_n|\leq 2$. Let $D_m$ be a connected component of the set $\{x\in\mathbb C^n:r_m(x)<0\}$, containing the point $z$.

Then there is $c>0$ such that $(D_m,z)$ and $(D,z)$ belong to $\mathcal D(c)$, $m>>1.$
\end{lem}

\begin{proof} Losing no generality assume that $D\Subset\mathbb B_n.$
Note that the conditions (1), (5), (6) of Definition \ref{30} are clearly satisfied. To find $c$ satisfying ($2$), we take $s>0$ such that $\mathcal H r (x;X)< s |X|^2$ for $x\in\ov\BB_n$ and $X\in(\mathbb R^{2n})_*$. Then $\HH r_m (x;X)<2s|X|^2$ for $x\in\ov\BB_n$, $X\in(\mathbb R^{2n})_*$ and $m>>1$. Let $U_0\subset\mathbb B_n$ be an open neighborhood of $\pa D$ such that $|\nabla r|$ is on $U_0$ between $3/4$ and $5/4$. Note that $\partial D_m\subset U_0$ and $|\nabla r_m|\in (1/2, 3/2)$ on $U_0$ for $m>>1$.

Fix $m$ and $a\in \partial D_m$ and put $b:=a-R\nu_{D_m}(a)$, where a small number $R>0$ will be specified later. There is $t>0$ such that $\nabla r_m(a)=2t(a-b)$. Note that $t$ may be arbitrarily large provided that $R$ was small enough. We take $t:=2s$ and $R:=|\nabla r_m(a)|/t$. Then we have $\mathcal H r_m(x;X)<2t |X|^2$ for $x\in\ov\BB_n$, $X\in(\mathbb R^{2n})_*$ and $m>>1$. Then a function $$h(x):=r_m(x)-t(|x-b|^2-R^2),\ x\in \mathbb C^n,$$ attains at $a$ its global maximum on $\ov\BB_n$ ($a$ is a strong local maximum and the Hessian of $h$ is negative on the convex set $\ov\BB_n$, cf. the proof of Lemma \ref{lemat}).
Thus $h\leq 0$ on $\mathbb B_n$. From this we immediately get (2).

Note that it follows from (2) that $D_m=\{x\in\mathbb C^n:r_m(x)<0\}$ for $m$ big enough (i.e. $\{x\in \mathbb C^n:\ r_m(x)<0\}$ is connected).

Moreover, the condition (2) implies the condition (3) as follows. We infer from Remark~\ref{D(c),4} that there is $c'>0$ such that $D$ satisfies (3) with $c'$. Let $m_0$ be such that the Hausdorff distance between $\partial D$ and $\partial D_m$ is smaller than $1/c'$ for $m\geq m_0$. There is $c''$ such that $D_{m_0}$ satisfies (3) with $c''$. Losing no generality we may assume that $c''<c'$. Take any $x,y\in D_m$. Since $D_m$ satisfies the interior ball condition with a radius $c$ we infer that there are balls of a radius $1/c$ contained in $D_m$ and containig $x$ and $y$ respectively. The centers of these balls lie in $D_{m_0}$. Using the fact that $(D_{m_0},z)$ lies in $\mathcal D(c'')$, we may join chosen centers with balls of a radius $1/(2c'')$ as in the condition (3), so we have found a chain consiting of balls of radii $c'$ and $c''$ joining $x$ and $y$.

Thus we may join $x$ and $y$ with balls contained entirely in the constructed chain whose radii depend only on $c'$ and $c''$.

Now we are proving $(4)$.  We shall show that there is $c>c'$ such that every $D_m$ satisfies (4) with $c$ for $m$ big enough. To do it let us cover $\partial D$ with a finite number of balls $B_j$, $j=1,\ldots,N$, from condition (4) and let $B'_j$ be a ball contained relatively in $B_j$ such that $\{B_j\}$ covers $\partial D$, as well. Let $\Phi_j$ be mappings corresponding to $B_j$. Let $\eps$ be such that any ball of radius $\eps$ intersecting $\partial D$ non-emptily is relatively contained in $B_j'$ for some $j$. Observe that any ball $B$ of radius $\eps/2$ intersecting non-emptily $\partial D_m$ is contained in a ball of radius $\eps$ intersecting non-emptily $\partial D$; hence it is contained in $B_j'$ for some $j$. Then the pair $B$, $\Phi_j$ satisfies the conditions (4) (b), (c) and (d). Therefore, it suffices to check that there is $c>2/\eps$ such that each pair $B_j'$, $\Phi_j$ satisfies the condition (4) for $D_m$ with $c$ ($m>>1$). This is possible since $\Phi_j(D_m)\subset\Phi_j(D)$, $D^\alpha\Phi_j(\pa D_m\cap B_j)$ converges to $D^\alpha\Phi_j(\pa D\cap B_j)$ for $|\alpha|\leq 2$ and for any $w\in\Phi(\pa D\cap B_j)$ there is a ball of radius $2/\eps$ containing $\Phi_j(D)$ and tangent to $\partial\Phi_j(D)$ at $w$. To be precise, we proceed as follows. 

Let $a,b\in\CC^n$ and let $x\in\pa B_n(a,\wi c)$, where $\wi c>c'$. Then a ball $B_n(2a-x,2\wi c)$ contains $B_n(a,\wi c)$ and is tangent to $B_n(a,\wi c)$ at $x$. There is a number $\eta=\eta(\delta,\wi c)>0$, independent of $a,b,x$, such that the diameter of the set $B_n(b,\wi c)\setminus B_n(2a-x,2\wi c)$ is smaller than $\delta>0$, whenever $|a-b|<\eta$ (this is a simple consequence of the triangle inequality).

Let $\wi s>0$ be such that $\mathcal H(r\circ\Phi_j^{-1})(x;X)\geq 2\wi s|X|^2$ for $x\in U_j$, $j=1,\ldots,N$, where $U_j$ is an open neighborhood of $\Phi_j(\partial D\cap B_j)$. Then, for $m$ big enough, $\mathcal H(r_m\circ \Phi_j^{-1})(x;X)\geq\wi s|X|^2$ for $x\in U_j$ and $\Phi_j(\partial D_m\cap B_j')\subset U_j$, $j=1,\ldots,N$. Repeating for the function $$x\longmapsto(r_m\circ\Phi_j^{-1})(x)-\wi t(|x-\wi b|^2-\wi R^2)$$ the argument used in the interior ball condition with suitable chosen $\wi t$ and uniform $\wi R>c$, we find that there is uniform $\wi\eps>0$ such that for any $j,m$ and $w\in\Phi_j(\partial D_m\cap B_j')$ there is a ball $B$ of radius $\wi R$, tangent to $\Phi_j(\partial D_m\cap B_j')$ at $w$, such that $\Phi_j(\partial D_m\cap B_j')\cap B_n(w,\wi\eps)\subset B$. Let $a_{j,m}(w)$ denote its center.

On the other hand for any $w\in \Phi_j(\partial D_m\cap B_j')$ there is $t>0$ such that $w'=w+t\nu (w)\in \Phi_j(\partial D\cap B_j)$, where $\nu(w)$ is a normal vector to $\Phi_j(\partial D_m\cap B_j')$ at $w$. Let $a_j(w')$ be a center of a ball of radius $\wi R$ tangent to $\Phi_j(\partial D\cap B_j)$ at $w'$. It follows that $|a_{j,m}(w)-a_j(w')|<\eta(\wi\eps/2,\wi R)$ provided that $m$ is big enough. 

Joinining the facts presented above, we finish the proof of the exterior ball condition (with a radius dependent only on $\wi\eps$ and $\wi R$).
\end{proof}

\begin{proof}[Proof of Theorems \ref{lem-car} and \ref{main} \emph{(}$\mathcal C^2$-smooth case$)$]
Losing no generality assume that $0\in D\Subset\BB_n$.

It follows from the Weierstrass Theorem that there is sequence $\{P_k\}$ of real polynomials on $\CC^n\simeq\mathbb R^{2n}$ such that $$D^{\alpha}P_{k}\to D^{\alpha}r \text{ uniformly on }\ov\BB_n,$$ where $\alpha=(\alpha_1,\ldots, \alpha_{2n})\in \mathbb N_0^{2n}$ is such that $|\alpha|=\alpha_1+\ldots +\alpha_{2n}\leq 2$. Consider the open set $$\wi D_{k,\eps}:=\{x\in \mathbb C^n:P_{k}(x)+\eps<0\}.$$ Let $\eps_{m}$ be a sequence of positive numbers converging to $0$ such that $3\eps_{m+1}<\eps_m.$

For any $m\in \mathbb N$ there is $k_{m}\in\NN$ such that $\sup_{\ov\BB_n}|P_{k_{m}}-r|<\eps_{m}$. Putting $r_{m}:=P_{k_{m}}+2\eps_{m}$, we get $r+\eps_{m}<r_{m}<r+3\eps_{m}$. In particular, $r_{m+1}<r_m.$

Let $D_m$ be a connected component of $D_{k_m,2\eps_m}$ containing $0$. It is a bounded strongly linearly convex domain with real analytic boundary and $r_m$ is its defining function provided that $m$ is big enough. Moreover, $D_{m}\subset D_{m+1}$ and $\bigcup_m D_{m}=D$. Using properties of holomorphically invariant functions and metrics we get Theorem~\ref{lem-car}.

We are left with showing the claim that for any different $z,w\in D$ (resp. $z\in D$, $v\in(\mathbb C^n)_*$) there is a weak $E$-mapping for $z,w$ (resp. for $z,v$). Fix $z\in D$ and $w\in D$ (resp. $v\in(\mathbb C^n)_*$). Then $z,w\in D_m$ (resp. $z\in D_m$), $m>>1$. Therefore, for any $m>>1$ one may find an $E$-mapping $f_m$ of $D_m$ for $z,w$ (resp. for $z,v$). Since $(D_m,z)\in \mathcal D(c)$ for some uniform $c>0$ ($m>>1$) (Lemma~\ref{un}), we find that $f_m$, $\wi f_m$ and $\rho_m$ satisfy the uniform estimates from Section~\ref{22}. Thus, passing to a subsequence we may assume that $\{f_m\}$ converges uniformly on $\CDD$ to a mapping $f\in\OO(\DD)\cap\cC^{1/2}(\CDD)$ passing through $z,w$ (resp. such that $f(0)=z$, $f'(0)=\lambda v$, $\lambda>0$), $\{\wi f_m\}$ converges uniformly on $\CDD$ to a mapping $\wi f\in\OO(\DD)\cap\mathcal C^{1/2}(\overline{\mathbb D})$ and $\{\rho_m\}$ is convergent uniformly on $\TT$ to a positive function $\rho\in\cC^{1/2}(\TT)$ (in particular, $f'\bullet\wi f=1$ on $\DD$, so $\wi f$ has no zeroes in $\CDD$). We already know that this implies that $f$ is a weak $E$-mapping of $D$.

To get $\cC^{k-1-\eps}$-smoothness of the extremal $f$ and its associated mappings for $k\geq 3$, it suffices to repeat the proof of Proposition~5 of \cite{Lem2}. This is just the Webster Lemma (we have proved it in the real analytic case --- see Proposition~\ref{6}). Namely, let $$\psi:\partial D\ni z\longmapsto(z,T_{D}^\mathbb{C}(z))\in \mathbb C^n\times(\mathbb P^{n-1})_*,$$ where $\mathbb P^{n-1}$ is the $(n-1)$-dimensional complex projective space. Let $\pi:(\CC^n)_*\longrightarrow\mathbb P^{n-1}$ be the canonical projection. 

By \cite{Web}, $\psi(\partial D)$ is a totally real manifold of $\mathcal C^{k-1}$ class. Observe that the mapping $(f,\pi\circ \wi f):\CDD\longrightarrow\CC^n\times\mathbb P^{n-1}$ is $1/2$-H\"older continuous, is holomorphic on $\mathbb D$ and maps $\mathbb T$ into $\psi(\partial D)$. Therefore, it is $\mathcal C^{k-1-\eps}$-smooth for any $\eps>0$, whence $f$ is $\mathcal C^{k-1-\eps}$-smooth. Since $\nu_D\circ f$ is of class $\mathcal C^{k-1-\eps}$, it suffices to proceed as in the proof of Proposition~\ref{6}.
\end{proof}

\section{Appendix}\label{Appendix}
\subsection{Totally real submanifolds}
Let $M\subset\CC^m$ be a totally real local $\CLW$ submanifold of the real dimension $m$. Fix a point $z\in M$. There are neighborhoods $U_0\su\RR^m$, $V_0\su\CC^m$ of $0$ and $z$ and a $\CLW$ diffeomorphism $\widetilde{\Phi}:U_0\longrightarrow M\cap V_0$ such that $\widetilde{\Phi}(0)=z$. The mapping $\widetilde{\Phi}$ can be extended in a natural way to a mapping $\Phi$ holomorphic in a neighborhood of $0$ in $\CC^m$. Note that this extension will be biholomorphic in a neighborhood of $0$. Actually, we have $$\frac{\partial\Phi_j}{\partial z_k}(0)=\frac{\partial\Phi_j}{\partial
x_k}(0)=\frac{\partial\widetilde{\Phi}_j}{\partial x_k}(0),\ j,k=1,\ldots,m,$$ where $x_k=\re z_k$. Suppose that the complex derivative $\Phi'(0)$ is not an isomorphism. Then there is $X\in(\CC^m)_*$ such that $\Phi'(0)X=0$, so \begin{multline*}0=\sum_{k=1}^m\frac{\partial\Phi}{\partial z_k}(0)X_k=\sum_{k=1}^m\frac{\partial\wi\Phi}{\partial x_k}(0)(\re X_k+i\im X_k)=\\=\underbrace{\sum_{k=1}^m\frac{\partial\wi\Phi}{\partial x_k}(0)\re X_k}_{=:A}+i\underbrace{\sum_{k=1}^m\frac{\partial\wi\Phi}{\partial x_k}(0)\im X_k}_{=:B}.\end{multline*}
The vectors $$\frac{\partial\wi\Phi}{\partial x_k}(0),\ k=1,\ldots,m$$ form a basis of $T^{\RR}_M(z)$, so $A,B\in T^{\RR}_M(z)$, consequently $A,B\in iT^{\RR}_M(z)$. Since $M$ is totally real, i.e. $T^{\RR}_M(z)\cap iT^{\RR}_M(z)=\{0\}$, we have $A=B=0$. By a property of the basis we get $\re X_k=\im X_k=0$, $k=1,\ldots,m$ --- a contradiction.

Therefore, $\Phi$ in a neighborhood of $0$ is a biholomorphism of two open subsets of $\CC^m$, which maps a neighborhood of $0$ in $\RR^m$ to a neighborhood of $z$ in $M$.

\begin{lemm}[Reflection Principle]\label{reflection}
Let $M\subset\CC^m$ be a totally real local $\CLW$ submanifold of the real
dimension $m$. Let $V_0\subset\CC$ be a neighborhood of $\zeta_0\in\TT$ and let $g:\overline{\DD}\cap V_0\longrightarrow\CC^m$ be a continuous mapping. Suppose that $g\in\OO(\DD\cap V_0)$ and $g(\TT\cap V_0)\subset M$. Then $g$ can be extended holomorphically past $\TT\cap V_0$.
\end{lemm}
\begin{proof}
In virtue of the identity principle it is sufficient to extend $g$ locally
past an arbitrary point $\zeta_0\in\TT\cap V_0$. For a point $g(\zeta_0)\in M$ take $\Phi$ as above. Let $V_1\subset V_0$ be a neighborhood of $\zeta_0$ such that $g(\CDD\cap V_1)$ is contained in the image
of $\Phi$. The mapping $\Phi^{-1}\circ g$ is holomorphic in $\DD\cap V_1$ and has
real values on $\TT\cap V_1$. By the ordinary Reflection Principle we can
extend this mapping holomorphically past $\TT\cap V_1$. Denote this extension by
$h$. Then $\Phi\circ h$ is an extension of $g$ in a neighborhood of $\zeta_0$.
\end{proof}

\subsection{Schwarz Lemma for the unit ball}
\begin{lemm}[Schwarz Lemma]\label{schw}
Let $f\in\OO(\DD,B_n(a,R))$ and $r:=|f(0)-a|$. Then $$|f'(0)|\leq \sqrt{R^2-r^2}.$$
\end{lemm}

\subsection{Some estimates of holomorphic functions of $\cC^{\alpha}$-class}

Let us recall some theorems about functions holomorphic in $\DD$ and continuous in $\CDD$. Concrete values of constants $M,K$ are possible to calculate, seeing on the proofs. In fact, it is only important that they do not depend on functions.
\begin{tww}[Hardy, Littlewood, \cite{Gol}, Theorem 3, p. 411]\label{lit1}
Let $f\in\OO(\DD)\cap\cC(\CDD)$. Then for $\alpha\in(0,1]$ the following conditions are equivalent
\begin{eqnarray}\label{47}\exists M>0:\ |f(e^{i\theta})-f(e^{i\theta'})|\leq M|\theta-\theta'|^{\alpha},\ \theta,\theta'\in\RR;\\
\label{45}\exists K>0:\ |f'(\zeta)|\leq K(1-|\zeta|)^{\alpha-1},\ \zeta\in\DD.
\end{eqnarray}
Moreover, if there is given $M$ satisfying \eqref{47} then $K$ can be chosen as $$2^{\frac{1-3\alpha}{2}}\pi^\alpha M\int_0^\infty\frac{t^\alpha}{1+t^2}dt$$ and if there is given $K$ satisfying \eqref{45} then $M$ can be chosen as $(2/\alpha+1)K$.
\end{tww}
\begin{tww}[Hardy, Littlewood, \cite{Gol}, Theorem 4, p. 413]\label{lit2}
Let $f\in\OO(\DD)\cap\cC(\CDD)$ be such that $$|f(e^{i\theta})-f(e^{i\theta'})|\leq M|\theta-\theta'|^{\alpha},\ \theta,\theta'\in\RR,$$ for some $\alpha\in(0,1]$ and $M>0$. Then $$|f(\zeta)-f(\zeta')|\leq K|\zeta-\zeta'|^{\alpha},\
\zeta,\zeta'\in\CDD,$$ where $$K:=\max\left\{2^{1-2\alpha}\pi^\alpha M,2^{\frac{3-5\alpha}{2}}\pi^\alpha\alpha^{-1} M\int_0^\infty\frac{t^\alpha}{1+t^2}dt\right\}.$$
\end{tww}
\begin{tww}[Privalov, \cite{Gol}, Theorem 5, p. 414]\label{priv}
Let $f\in\OO(\DD)$ be such that $\re f$ extends continuously on $\CDD$ and $$|\re f(e^{i\theta})-\re f(e^{i\theta'})|\leq M|\theta-\theta'|^\alpha,\ \theta,\theta'\in\RR,$$ for some $\alpha\in(0,1)$ and $M>0$. Then $f$ extends continuously on $\CDD$ and $$|f(\zeta)-f(\zeta')|\leq K|\zeta-\zeta'|^\alpha,\ \zeta,\zeta'\in\CDD,$$ where $$K:=\max\left\{2^{1-2\alpha}\pi^\alpha,2^{\frac{3-5\alpha}{2}}\pi^\alpha\alpha^{-1}\int_0^\infty\frac{t^\alpha}{1+t^2}dt\right\}\left(\frac{2}{\alpha}+1\right)2^{\frac{3-3\alpha}{2}}\pi^{\alpha}M\int_0^\infty\frac{t^\alpha}{1+t^2}dt.$$
\end{tww}
\subsection{Sobolev space}
The Sobolev space $W^{2,2}(\TT)=W^{2,2}(\TT,\CC^m)$ is a space of functions $f:\TT\longrightarrow\CC^m$, whose first two derivatives (in the sense of distribution) are in $L^2(\TT)$ (here we use a standard identification of functions on the unit circle and functions on the interval $[0,2\pi]$). Then $f$ is $\mathcal C^1$-smooth.

It is a complex Hilbert space with the following scalar product
$$\langle f,g\rangle_W:=\langle f,g\rangle_{L}+\langle f',g'\rangle_{L}+\langle f'',g''\rangle_{L},$$
where $$\langle\wi f,\wi g\rangle_{L}:=\frac{1}{2\pi}\int_0^{2\pi}\langle\wi f(e^{it}),\wi g(e^{it})\rangle dt.$$ Let $\|\cdot\|_L$, $\|\cdot\|_W$ denote the norms induced by $\langle\cdotp,-\rangle_L$ and $\langle\cdotp,-\rangle_W$. The following characterization simply follows from Parseval's identity $$W^{2,2}(\TT)=\left\{f\in L^2(\TT):\sum_{k=-\infty}^{\infty}(1+k^2+k^4)|a_k|^2<\infty\right\},$$ where $a_k\in\CC^m$ are the $m$-dimensional Fourier coefficients of $f$, i.e. $$f(\zeta)=\sum_{k=-\infty}^{\infty}a_k\zeta^k,\ \zeta\in\TT.$$ More precisely, Parseval's identity gives $$\|f\|_W=\sqrt{\sum_{k=-\infty}^{\infty}(1+k^2+k^4)|a_k|^2},\ f\in W^{2,2}(\TT).$$ Note that $W^{2,2}(\TT)\su\mc{C}^{1/2}(\TT)\su\mc{C}(\TT)$ and both inclusions are continuous (in particular, both inclusions are real analytic). Note also that
 \begin{equation}\label{67}\|f\|_{\sup}\leq\sum_{k=-\infty}^{\infty}|a_k|\leq\sqrt{\sum_{k=-\infty}^{\infty}\frac{1}{1+k^2}\sum_{k=-\infty}^{\infty}(1+k^2)|a_k|^2}\leq\frac{\pi}{\sqrt 3}\|f\|_W.\end{equation}\\

Now we want to show that there exists $C>0$ such that $$\|h^\alpha\|_W\leq C^{|\alpha|}\|h_1\|^{\alpha_1}_W\cdotp\ldots\cdotp\|h_{2n}\|^{\alpha_{2n}}_W,\quad h\in W^{2,2}(\TT,\CC^n),\,\alpha\in\NN_0^{2n}.$$ Thanks to the induction it suffices to prove that there is $\wi C>0$ satisfying $$\|h_1h_2\|_W\leq\wi C\|h_1\|_W\|h_2\|_W,\quad h_1,h_2\in W^{2,2}(\TT,\CC).$$ Using \eqref{67}, we estimate $$\|h_1h_2\|^2_W=\|h_1h_2\|^2_L+\|h_1'h_2+h_1h_2'\|^2_L+\|h_1''h_2+2h_1'h_2'+h_1h_2''\|^2_L\leq$$$$\leq C_1\|h_1h_2\|_{\sup}^2+(\|h_1'h_2\|_L+\|h_1h_2'\|_L)^2+(\|h_1''h_2\|_L+\|2h_1'h_2'\|_L+\|h_1h_2''\|_L)^2\leq$$\begin{multline*}\leq C_1\|h_1\|_{\sup}^2\|h_2\|_{\sup}^2+(C_2\|h_1'\|_L\|h_2\|_{\sup}+C_2\|h_1\|_{\sup}\|h_2'\|_L)^2+\\+(C_2\|h_1''\|_L\|h_2\|_{\sup}+C_2\|2h_1'h_2'\|_{\sup}+C_2\|h_1\|_{\sup}\|h_2''\|_L)^2\leq\end{multline*}\begin{multline*}\leq C_3\|h_1\|_W^2\|h_2\|_W^2+(C_4\|h_1\|_W\|h_2\|_W+C_4\|h_1\|_W\|h_2\|_W)^2+\\+(C_4\|h_1\|_W\|h_2\|_W+2C_2\|h_1'\|_{\sup}\|h_2'\|_{\sup}+C_4\|h_1\|_W\|h_2\|_W)^2\leq\end{multline*}$$\leq C_5\|h_1\|_W^2\|h_2\|_W^2+(2C_4\|h_1\|_W\|h_2\|_W+2C_2\|h_1'\|_{\sup}\|h_2'\|_{\sup})^2$$ with constants $C_1,\ldots,C_5$. Expanding $h_j(\zeta)=\sum_{k=-\infty}^{\infty}a^{(j)}_k\zeta^{k}$, $\zeta\in\TT$, $j=1,2$, we obtain $$\|h_j'\|_{\sup}\leq\sum_{k=-\infty}^{\infty}|k||a^{(j)}_k|\leq\sqrt{\sum_{k\in\ZZ_*}\frac{1}{k^2}\sum_{k\in\ZZ_*}k^4|a^{(j)}_k|^2}\leq\frac{\pi}{\sqrt 3}\|h_j\|_W$$ and finally $\|h_1h_2\|^2_W\leq C_6\|h_1\|_W^2\|h_2\|_W^2$ for some constant $C_6$.
\subsection{Matrices}
\begin{propp}[Lempert, \cite{Lem2}, Th\'eor\`eme $B$]\label{12}
Let $A:\TT\longrightarrow\CC^{n\times n}$ be a matrix-valued real analytic mapping such
that $A(\zeta)$ is self-adjoint and strictly positive for any $\zeta\in\TT$. Then there exists $H\in\OO(\CDD,\CC^{(n-1)\times(n-1)})$ such that $\det H\neq 0$ on $\CDD$ and $HH^*=A$ on $\TT$.
\end{propp}
In \cite{Lem2}, the mapping $H$ was claimed to be real analytic in a neighborhood of $\CDD$ and holomorphic in $\DD$, but it is equivalent to $H\in\OO(\CDD)$. Indeed, since $\ov\pa H$ is real analytic near $\CDD$ and $\ov\pa H=0$ in $\DD$, the identity principle for real analytic functions implies $\ov\pa H=0$ in a neighborhood of $\CDD$.
\begin{propp}[\cite{Tad}, Lemma $2.1$]\label{59}
Let $A$ be a complex symmetric $n\times n$ matrix. Then $$\|A\|=\sup\{|z^TAz|:z\in\CC^n,\,|z|=1\}.$$
\end{propp}
\bigskip
\textsc{Acknowledgements.} We would like to thank Sylwester Zaj\k ac for helpful discussions. We are also grateful to our friends for the participation in preparing some parts of the work.

\end{document}